\documentclass[11pt]{article}
\usepackage{geometry}                
\usepackage{graphicx}
\usepackage{amssymb}
\usepackage{epstopdf}
\usepackage{amsmath,amssymb,amsthm}
\usepackage{amsfonts}
\usepackage[numbers,sort]{natbib}
\usepackage[retainorgcmds]{IEEEtrantools}

\newtheorem{thm}{Theorem}[section]
\newtheorem{prop}[thm]{Proposition}

\newtheorem{lem}[thm]{Lemma}

\title{On the Holomorphy of Exterior-Square $L$-functions\\Preprint}
\author{Dustin Belt}

\begin{document}
\maketitle

\begin{abstract}

In this paper, we show that the twisted partial exterior-square $L$-function has a meromorphic continuation to the whole complex plane with only two possible simple poles at $s=1$ and $s=0$. We do this by establishing the nonvanishing of the local zeta integrals defined by Jacquet and Shalika for any fixed $s_0$. The even case is treated in detail. The odd case is treated briefly, in which case, the $L$-function is shown to be entire.
\end{abstract}

\maketitle

\section{Introduction}

Let $\pi$ be an irreducible automorphic cuspidal representation of $GL_r(\mathbb{A}_F)$ where $F$ is a number field. Then the connected component of the corresponding $L$-group, ${}^L{GL}_r^{\circ}$ is simply the group $GL_r(\mathbb{C})$. Let $\rho$ be the standard representation of degree $r$ of ${}^L{GL}_r^{\circ}$. We have the standard $L$-function attached to $\pi$, $L(s,\pi,\rho)$, and the so-called convolution of this $L$-function with itself $L(s,\pi,\rho\otimes\rho)$, which we will also denote as $L(s,\pi\times\pi)$. 

The representation $\rho\otimes\rho$ decomposes into the direct sum $$\rho\otimes\rho=\mathrm{Sym}^2\oplus {\bigwedge}^2$$ where $\mathrm{Sym}^2$ is the symmetric square representation of $GL_r(\mathbb{C})$ on the space of symmetric tensors and ${\bigwedge}^2$ is the exterior square representation of $GL_r(\mathbb{C})$ on the space of antisymmetric tensors. Then the $L$-function $L(s,\pi,\rho\otimes\rho)$ decomposes as the product $$L(s,\pi,\rho\otimes\rho)=L(s,\pi,\mathrm{Sym}^2)L(s,\pi,{\bigwedge}^2).$$ 
It is known (see  \cite{prod1} or Theorem 10.1.1 of \cite{shahidinewbook}) that the $L$-function on the left has a simple pole at $s=1$ if and only if $\pi$ is self-dual (i.e. $\tilde{\pi}=\pi$ where $\tilde{\pi}$ is the contragredient of $\pi$). In this case, one of the two functions on the right would also have a pole at $s=1$. They cannot both have a pole at $s=1$ as this would imply a double pole on the left, which is not the case. It is known that neither function on the right vanishes for $\Re(s)=1$ \cite{shahidi1981}. Thus, $L(s,\pi,\rho\otimes\rho)$ has a simple pole at $s=1$ if and only if one (but not both) of $L(s,\pi,\mathrm{Sym}^2)$ or $L(s,\pi,{\bigwedge}^2)$ has a pole at $s=1$. Which of the two has the pole is closely related to the lifting theory of automorphic representations predicted by the Langlands functoriality conjecture. For example, it is known that $L(s,\pi,{\bigwedge}^2)$ has this pole if and only if $\pi$ is the functorial lift of a generic cuspidal representation of $SO_{2n+1}(F)$ \cite{arthur}.

In 1990, Herv{\'e} Jacquet and Joseph Shalika \cite{jasha} showed that the exterior square $L$-function $L(s,\pi,{\bigwedge}^2\otimes \chi)$, twisted by a idele class character $\chi$, has a pole at $s=1$ if and only if  $r$ is even and certain period integrals are non-zero. In that paper, the main technique is in giving a certain integral representation for $L(s,\pi,{\bigwedge}^2\otimes\chi)$, or rather the partial version of it, which can be decomposed as a product of local integrals. These local integrals are shown to converge for $\Re(s)\geq 1-\eta$ for a small positive number $\eta$ depending on the representation $\pi$, and the local data can be chosen so the local integral is non-zero at $s=1$ (See Propositions 1 and 3 of Section 7 of \cite{jasha}). Having this, the poles of the $L$-function can be related to the poles of a certain Eisenstein series. 

However, the method used in \cite{jasha} to show that the local integrals are nonvanishing for $s=1$ is very dependent on the absolute convergence of the integrals involved, which can only be assumed when $\Re(s)\geq 1$. Thus, while it is not difficult to generalize the proofs so that the desired results hold for any $s$ in the half plane $\Re(s)\geq 1$, without an improvement on the region of convergence for the local integrals, one cannot generalize the result to other $s$ without first developing meromorphic continuations for each of the various integrals involved.

Here, we expand on the results of \cite{jasha}, establishing that, in the case $r$ is even, $s=1$ and $s=0$ are the only possible locations for poles of the twisted partial $L$-function $L^S(s,\pi,{\bigwedge}^2\otimes \chi)$ for some  idele class character $\chi$ of absolute value 1. See Theorem \ref{global}. Our general strategy is the same as that of \cite{jasha}.

In Section 2, we show that the local Jacquet-Shalika $J(s,\chi,\Phi,W)$ integral can be written as 
\begin{equation}\int_{P_n\backslash GL_n} J_1(s,h,\chi,W)\Phi(\epsilon_n h)dh\label{summary} \end{equation}
where $J_1$ is a related integral, and $P_n$ is the mirabolic subgroup of $GL_n$. We refer to the body of the text for the precise definitions. This is essentialy the first step of the descending induction argument used in \cite{jasha}.
In Sections 3 and 4, rather than make further reductions, we use a result from \cite{jacquetnew} which allows us to choose $W$ so that, on the domain of integration in $J_1$, $W$ behaves like a smooth function of compact support, so that $J_1$ is absolutely convergent for all $s$, and thus entire. Furthermore, for fixed $s_0$, $W$ can be further chosen so that $J_1(s_0,1_n,\chi, W)\neq 0$ (see Lemma \ref{J1}). In the nonarchimedean case, since $J_1$ is locally constant as a function of $h$, this is sufficient to show that $\Phi$ can be chosen so that the integral (\ref{summary}) is entire and non-zero at $s=s_0$.

However, in the archimedean case, a question arises as to the nature of $J_1(s,h,\chi,W)$ for $h\neq 1_n$. If $W$ is chosen so that $J_{1}(s,1_{n},\chi,W)$ is entire, we cannot say that the integral $J_{1}(s,h,\chi,W)$ with $h\neq 1_{n}$ is absolutely convergent for all $s$. Thus, the expression (\ref{summary}) only makes sense for all $s$ if we have a meromorphic continuation of $J_1$ to the whole complex plane. We establish this meromorphic continuation, and do so in such a way so as to show that the dependence on $h$ is continuous in a sense which is made precise in Section 3. Thus, (\ref{summary}) defines a meromorphic continuation of $J(s,\chi,\Phi,W)$ to the whole complex plane and $\Phi$ can be chosen so that (\ref{summary}) is meromorphic with respect to $s$ and non-zero when $s=s_0$. See Theorem \ref{main}.
In section 5, this non-vanishing is then used to establish the desired holomorphy of the partial $L$-function (Theorem \ref{global}).

This paper was motivated by the work of Asgari and Shahidi \cite{twistednormalized} and was originally considered in connection with their current research concerning the functorial transfer of generic, automorphic representations from the quasi-split general spin groups to general linear groups \cite[Proposition 5.2]{asgarishahidi}. This paper is also related to the recent work of Pramod Kewat and Ravi Raghunathan \cite{pramod}.

We hope to follow up this paper by proving a similar holomorphy result concerning full (or completed) twisted exterior square $L$-functions. It appears that, with the present state of the local theory within the approach of \cite{jasha}, this will be best done using ideas more along the lines of Arthur \cite{arthur}, Moeglin and Waldspurger \cite{mogwald}, Shahidi \cite{shahidinewbook} and Cogdell, Kim, Piatetski-Shapiro and Shahidi \cite{normalized}, at least for the case $\chi=1$. For other $\chi$, we hope to build on the work of Asgari and Shahidi \cite{twistednormalized} and Hundley and Sayag \cite{hundleysayag}, and possible generalizations of \cite{arthur}. When this is completed, all of the desired analytic properties predicted by Langland's Conjecture in this particular case will have been established. 

In Section 2, we set up the local problem and present some of the initial manipulations of the integrals.
In Section 3, we treat the archimedean case, developing a meromorphic continuation for the Jacquet-Shalika integrals and establishing the desired nonvanishing result for this meromorphic continuation. We do this by modifying the results in \cite[Section 4]{jasha} on gauge representations of Whittaker functions and considering the generalized Mellin Transforms of Schwartz-Bruhat functions.
In Section 4, we treat the nonarchimedean case. Rather than generalize the work of Jacquet and Shalika \cite{jasha}, we modify the technique developed in Section 3, and establish the desired nonvanishing without appealing to a meromorphic continuation, showing, rather, that the Jacquet-Shalika integrals can be taken to be holomorphic functions in $s$. The work here is simpler in many ways, and does not make use of stronger results such as are used in Section 3.
In Section 5, we use the results of the previous sections to prove the desired global result, following the strategy used in section 8 of Jacquet and Shalika \cite{jasha}.

Throughout the paper, we treat the case where $r$ is even. However, in Chapter 6, we briefly treat the odd case, where the $L$-function is shown to be entire, using simpler versions of the techniques used in the even case. In \cite{kimcan}, the completed $L$-function for this case is shown to be entire.

The author would like to thank Freydoon Shahidi for his insight and advice. Also, Stephen D. Miller for his comments and for pointing out the 2010 paper by Jacquet, and Jiu-Kang Yu and David Goldberg for their helpful suggestions and comments. Also, thank you to James Cogdell for his careful reading of previous versions of this paper and his many useful comments.

\section{Local Considerations}

\subsection{Statement of the Main Local Result}

Let $F$ be a local field, $\psi$ a non-trivial additive character of $F$, $r=2n$ an even integer, $\pi$ a unitary irreducible representation of $GL_r(F)$, and $\chi$ a unitary character.

For a linear algebraic group $G$, we will often abuse the notation by writing $G$ for $G(F)$ whenever the context is clear, such as in the domain of integration of an integral.

Let $N_r$ be the unipotent subgroup of $GL_r$ consisting of upper triangular matrices with ones along the diagonal and let $\theta_r$ denote the character of $N_r$ defined by 
$$\theta_r\left[\left(\begin{array}{ccccc}1&u_{1,2}&\cdots&*&*\\0&1&\ddots&\vdots &\vdots\\ \vdots &\vdots& \ddots &1&u_{r-1,r}\\0&0&\cdots&0&1\end{array}\right)\right]=\psi\left(\sum_{i=1}^{r-1}u_{i,i+1}\right).$$

We denote by $\mathcal{W}(\pi,\psi)$ the Whittaker model of $\pi$ associated to this character. Let $W\in\mathcal{W}(\pi,\psi)$ and let $\Phi$ be a Schwartz-Bruhat function in $n$ variables.

Let $\sigma$ be the $2n\times 2n$ permutation matrix corresponding to the permutation changing the sequence $$(1,2,3,\ldots,n,n+1,n+2,\ldots,2n)$$ into the sequence $$(1,3,5,\ldots,2n-1,2,4,6,\ldots, 2n).$$ 

Let $M_n$ be the space of $n\times n$ matrices, $\mathfrak{p}_{0,n}$ the space of upper triangular matrices in $M_n$.

Consider the integral $J=J(s,\chi,W,\Phi)$ defined by \begin{equation}J=\int W\left[\sigma\left(\begin{array}{cc}1_n&Z\\0&1_n\end{array}\right)\left(\begin{array}{cc} g &0\\0&g\end{array}\right)\right]\psi(-\textrm{Tr}(Z))dZ\,\Phi(\epsilon_n g)\chi(\det g)|\det g|^sdg.\label{Jdef}\end{equation}

Here, $Z$ is integrated over the quotient $\mathfrak{p}_{0,n}(F)\backslash M_n(F),$ $g$ over the quotient $N_n(F)\backslash GL_n(F),$ and we have set $$\epsilon_n=(\underbrace{0,0,\ldots,0}_{n-1},1).$$
Note that, for $g\in GL_n(F)$, $\epsilon_n g$ gives the bottom row of $g$ as a row vector.

In the following, we will imbed $GL_{m-1}$ into $GL_m$ in the usual way as the upper left block.

Proposition 1 in section 7 of \cite{jasha} states the following:

\begin{prop}\label{absconv}Given $\pi$, there is an $\eta>0$ such that the integral $J$ converges absolutely for $\Re(s)>1-\eta$.\end{prop}

We wish to prove the following theorem:

\begin{thm}\label{main}Suppose $F$ is a local field. Fix a complex number $s_0$. There exist $\Phi$ and $W$, depending on $s_0$, such that $J(s,\chi,\Phi,W)$ extends meromorphically to the whole complex plane, and $$J(s_0,\chi,\Phi,W)\neq 0,$$ possibly with $J(s,\chi,\Phi, W)$ having a pole at $s=s_0$ if $F$ is archimedean. If $F$ is nonarchimedean, then $\Phi$ and $W$ can be chosen so that $J(s,\chi,\Phi,W)$ is entire.\end{thm}

\subsection{Preliminary Considerations}

Assume, for the moment, that $\Re(s)\gg0$, so that the integrals above converge absolutely, and that $F$ is any local field.

Let $P_n$ denote the mirabolic subgroup of $GL_n$, i.e., the subgroup consisting of matrices of the form
$$\left(\begin{array}{cc} g&u\\0&1\end{array}\right)$$ where $g\in GL_{n-1}$ and $u\in F^{n-1}$ as a column vector.

We can choose $\Phi$ to be compactly supported away from 0. Then $g\mapsto \Phi(\epsilon_n g)$ is arbitrary among those functions which are smooth on $GL_n$, invariant on the left under the subgroup $P_n$, and of compact support modulo that subgroup.

Then we can decompose $N_n\backslash GL_n$ as a product of quotient spaces $$N_n\backslash GL_n\simeq \left(N_{n-1}\backslash GL_{n-1}\right)\cdot \left( P_n\backslash GL_n\right).$$
Using this decomposition, we can write the integral $J(s,\chi,\Phi,W)$ as 
\begin{multline}\int_{P_n\backslash GL_n}\int_{N_{n-1}\backslash GL_{n-1}}\int_{\mathfrak{p}_{0,n}\backslash M_n}W\left[\sigma\left(\begin{array}{cc}1_n&Z\\0&1_n\end{array}\right)\left(\begin{array}{cccc} g &0&0&0\\0&1&0&0\\0&0&g&0\\0&0&0&1\end{array}\right)\left(\begin{array}{cc} h &0\\0&h\end{array}\right)\right] \\
\times \psi(-\textrm{Tr}(Z))dZ\,\chi(\det g)|\det g|^{s-1}dg\,\chi(\det h)|\det h|^{s}\Phi(\epsilon_n h)dh.\label{one}\end{multline}

Remark: Note that the cosets of $ P_n\backslash GL_n$ are completely determined by the bottom row of the matrix $h$, i.e., by $\epsilon_n h$. In this case, $\Phi(\epsilon_n \cdot) $ is a smooth function on this homogeneous space.

For $h\in GL_n$ set \begin{IEEEeqnarray}{rCl}J_1(s,h,\chi,W)&=& 
\chi(\det h)|\det h|^{s-1}\nonumber \\
 &&\times  \int_{N_{n-1}\backslash GL_{n-1}}\int_{\mathfrak{p}_{0,n}\backslash M_n}W\left[\sigma\left(\begin{array}{cc}1_n&Z\\0&1_n\end{array}\right)\left(\begin{array}{cccc} g &0&0&0\\0&1&0&0\\0&0&g&0\\0&0&0&1\end{array}\right)\left(\begin{array}{cc} h &0\\0&h\end{array}\right)\right]
 \nonumber \\ && \times \psi(-\textrm{Tr}(Z))dZ\,\chi(\det g)|\det g|^{s-1}dg.\label{j1}\end{IEEEeqnarray} 

Note that the proof of Proposition \ref{absconv} (see \cite{jasha}) shows that this integral converges absolutely for $\Re(s)\gg0$. As a function of $h$, $J_1(s,h,\chi,W)$ is invariant on the left by elements of $P_n$, so we will often treat it as a function of the homogeneous space $P_n\backslash GL_n$.

We have, for $\Re(s)\gg 0$:

\begin{equation}J(s,\chi,\Phi,W)=\int_{P_n\backslash GL_n} J_1(s,h,\chi,W)\Phi(\epsilon_n h)dh.\label{jj1real}\end{equation}

We will prove the following lemma concerning the function $J_1$.

\begin{lem}\label{J1} Let $s_0$ be a complex number. There exists $W\in\mathcal{W}(\pi,\psi)$ such that $J_1(s,1_n,\chi,W)$ converges absolutely for all $s$ and $$J_1(s_0,1_n,\chi,W)\neq 0.$$\end{lem}

We will use the following result due to Gel{$'$}fand and Kajdan \cite{GK75} and Jacquet \cite{jacquetnew}:

\begin{prop}\label{smoothvector} Let $C_c^\infty(\theta_{r-1},GL_{r-1})$ be the space of smooth functions $\phi$ on $GL_{r-1}$, compactly supported  modulo $N_{r-1}$, such that $\phi(ug)=\theta_{r-1}(u)\phi(g)$ for $u\in N_{r-1}$, $g\in GL_{r-1}$. For every $\phi_0\in C_c^\infty(\theta_{r-1},GL_{r-1})$ there is a unique $W_{\phi_0}\in \mathcal{W}(\pi,\psi)$ such that for every $g\in GL_{r-1}$ $$W_{\phi_0}\left[\left(\begin{array}{cc}g&0\\0&1\end{array}\right)\right]=\phi_0(g).$$

Furthermore, in the case $F=\mathbb{R}$ or $\mathbb{C}$, the map $\phi_0\mapsto W_{\phi_0}$ is continuous.

\end{prop}

Remark: This, of course, holds for $r$ an odd integer as well. 

\begin{proof}

The nonarchimedean case is due to Gel{$'$}fand and Kajdan \cite{GK75}. See also Lemma 1, Section 2 of \cite{jacquetnew}.

For $F=\mathbb{R}$, this is Proposition 5 of section 3 in \cite{jacquetnew}. The arguments used in \cite{jacquetnew} also apply to $F=\mathbb{C}$. 

\end{proof}

\section{The Archimedean Case}

Assume throughout this section that $F=\mathbb{R}$ or $\mathbb{C}$.

\subsection{Some Technical Results}

The bulk of the work in establishing Lemma \ref{J1} in the archimedean case is found in the following technical lemmas. While similar results can be obtained for nonarchimedean fields, we will restrict ourselves to the archimedean case.

For $0\leq l\leq n$ and $\Phi$ a Schwartz-Bruhat function on $F^l$, set \begin{equation}W_1(g)=\int_{F^l}W\left[g\left(\begin{array}{cccccc} 1_l&0&0&0&0&0\\0&1&0&0&0&0\\0&0&1_{n-l-1}&0&0&0\\0&0&0&1_l&u&0\\0&0&0&0&1&0\\0&0&0&0&0&1_{n-l-1}\end{array}\right)\right]\Phi(u)du.\label{w1}\end{equation}

Note that this integral is absolutely convergent. 
Since $F$ is archimedean, recall (see, for example, \cite{jacquet}) that there is an integer $N>0$ and for $W\in \mathcal{W}(\pi,\psi)$, a constant $C$ such that $$|W(g)|\leq C||g||^N.$$ Keeping this bound in mind, integrating against a Schwartz-Bruhat function $\Phi$ results in a convergent integral.

Recall that $W$ transforms under the character $$\theta_r\left(\begin{array}{ccccc}1&u_1&\cdots&\ast& \ast\\0&1&\cdots&\ast&\ast\\ \vdots& \vdots& \ddots& \vdots&\vdots\\ 0& 0 &\cdots & 1&u_{2n-1}\\ 0& 0& \cdots& 0 &1\end{array}\right)=\psi\left(\sum_{j=1}^{2n-1} u_j\right),$$ of $N_{r}$. 

So, for $n\in N_{r}$, we have $W_1(ng)=\theta_r(n)W_1(g).$

\begin{lem}\label{dixmier1} Let $F=\mathbb{R}$ or $\mathbb{C}$, $\psi$ a non-trivial additive character, $r=2n$ an even integer, and $\pi$ a unitary irreducible representation of $GL_r(F)$ with Whittaker model $\mathcal{W}(\pi,\psi)$.
\begin{itemize}
\item[(a)] Given $W\in\mathcal{W}(\pi,\psi)$ and a Schwartz-Bruhat function $\Phi$ on $F^l$, the function $W_1$ defined by $$W_1(g)=\int_{F^l}W\left[g\left(\begin{array}{cccccc} 1_l&0&0&0&0&0\\0&1&0&0&0&0\\0&0&1_{n-l-1}&0&0&0\\0&0&0&1_l&u&0\\0&0&0&0&1&0\\0&0&0&0&0&1_{n-l-1}\end{array}\right)\right]\Phi(u)du$$ is again in $\mathcal{W}(\pi,\psi)$.

\item[(b)] Given $W\in\mathcal{W}(\pi,\psi)$, $W$ can be written as a finite sum 
$$W(g)=\sum_{j} \int_{F^l}W_j\left[g\left(\begin{array}{cccccc} 1_l&0&0&0&0&0\\0&1&0&0&0&0\\0&0&1_{n-l-1}&0&0&0\\0&0&0&1_l&u&0\\0&0&0&0&1&0\\0&0&0&0&0&1_{n-l-1}\end{array}\right)\right]\Phi_j(u)du$$
where the $\Phi_j$ are Schwartz-Bruhat functions on $F^l$ and the $W_j$ are in $\mathcal{W}(\pi,\psi)$.
\end{itemize}
\end{lem}

\begin{proof} For part (a), if $W(g)=\lambda(\pi(g)v)$, then $W_1(g)=\lambda(\pi(g)v_1)$ where $$v_1=\int_{F^l}\pi\left(\begin{array}{cccccc} 1_l&0&0&0&0&0\\0&1&0&0&0&0\\0&0&1_{n-l-1}&0&0&0\\0&0&0&1_l&u&0\\0&0&0&0&1&0\\0&0&0&0&0&1_{n-l-1}\end{array}\right)\Phi(u)vdu.$$ 

Since $\Phi$ is a Schwartz-Bruhat function, $v_1$ is a $C^\infty$ vector by Lemma 2.6 of \cite{wallach}.
In fact, if we appeal to the Dixmier-Malliavin Lemma \cite{dixmier} as stated at the beginning of section 6 of \cite{jacquet}, any vector $v$ can be written as a finite sum $$v=\sum_i\int_{F^l}\pi\left(\begin{array}{cccccc} 1_l&0&0&0&0&0\\0&1&0&0&0&0\\0&0&1_{n-l-1}&0&0&0\\0&0&0&1_l&u&0\\0&0&0&0&1&0\\0&0&0&0&0&1_{n-l-1}\end{array}\right)v_i\Phi_i(u)du$$
where the $\Phi_i$ are Schwartz-Bruhat functions (specifically, they can be taken to be smooth functions of compact support). This proves part (b) in the archimedean case.

\end{proof}

For $0\leq l\leq n$, $g\in GL_n(F)$ and $W\in\mathcal{W}(\pi,\psi)$, we define 
$$\Lambda_l(W,g)=\int_{\mathfrak{p}_{0,l}\backslash M_{l}}W\left[\sigma\left(\begin{array}{cccc}1_{l}&0&Z&0\\0&1_{n-l}&0 &0\\ 0&0&1_{l}&0\\0&0&0&1_{n-l}\end{array}\right)\left(\begin{array}{cc} g &0\\0&g\end{array}\right)\right]\psi(-\textrm{Tr}(Z))dZ.$$

Using the notation of Jacquet \cite{jacquetfest}, for $a_i\in F^\times$, we set $$\textrm{diar}[a_1,a_2,\ldots,a_{n-1}]=\textrm{diag}(a_1a_2\cdots a_{n-1},a_2\cdots a_{n-1},\ldots, a_{n-1},1).$$

Let $K_n$ denote the maximal compact subgroup of $GL_n(F)$.

By $\mathcal{S}(F^{n-1}\times K_n)$ we mean the space of Schwartz-Bruhat functions on $F^{n-1}\times K_n$. That is, the space of smooth, rapidly decreasing functions $\phi(a_1,a_2,\ldots,a_{n-1},k)$ on $F^{n-1}$, depending continuously on $k\in K_n$ (See the remark at the end of section 3.11 of \cite{jasha}).

\begin{lem}\label{gauge}Let $F=\mathbb{R}$ or $\mathbb{C}$. Fix $0\leq l\leq n$. For each $j$, $0\leq j\leq n-1$, there exists a finite set $C_j$ of characters and for each $\chi\in C_j$, an integer $b_\chi$ with the following 
property: let $X_j$ be the set of finite functions on $F^\times$ of the form $\chi(a)(\log|a|)^b$ with $\chi \in C_j$ and $b\leq b_\chi$, and let $X$ be the set of finite 
functions on $F^{n-1}$ which are products of functions in the sets $X_j$. Then for any $W\in \mathcal{W}(\pi,\psi)$, there are Schwartz-Bruhat functions $\phi_\xi$ in $\mathcal{S}(F^{n-1}\times K_n)$ such that, for $a=\textrm{diar}[a_1,a_2,\ldots,a_{n-1}]$ and $k\in K_n$
 \begin{equation}\Lambda_l(W,ak)=\sum_{\xi\in X}\phi_\xi(a_1,a_2,\ldots,a_{n-1},k)\xi(a_1,a_2,\ldots,a_{n-1}).\end{equation}

\end{lem}

\begin{proof} For now, we drop the dependence on $k$. 
Consider $\Lambda_l(W,a)$ with $a=\textrm{diar}[a_1,\ldots,a_{n-1}]$

Note that, in the case $l=0$,
 $$\Lambda_0(W,a)=W\left[\sigma \left(\begin{array}{cc}a&0\\0&a\end{array}\right)\right]=W[b\sigma]$$ 
where $$b=\textrm{diag}(a_1 a_2\cdots a_{n-1},a_1 a_2\cdots a_{n-1},a_2\cdots a_{n-1},a_2\cdots a_{n-1},\ldots,a_{n-1},a_{n-1},1,1)$$
$$=\textrm{diar}[1,a_1,1,a_2,\ldots, a_{n-1},1]$$
and the assertion follows directly from Proposition 3, Section 4 of \cite{jasha}.

Now, assume the lemma holds for $\Lambda_{l}$. We wish to show that it also holds for $\Lambda_{l+1}$.

We write $$\left(\begin{array}{cccc}1_{l+1}&0&Z&0\\0&1_{n-l-1}&0 &0\\ 0&0&1_{l+1}&0\\0&0&0&1_{n-l-1}\end{array}\right)=\left(\begin{array}{cccccc} 1_l&0&0&Z'&0&0\\0&1&0&Y&0&0\\0&0&1_{n-l-1}&0&0&0\\0&0&0&1_l&0&0\\0&0&0&0&1&0\\0&0&0&0&0&1_{n-l-1}\end{array}\right)$$
 where $Z'\in \mathfrak{p}_{0,l}\backslash M_{l}$ and $Y$ is a row vector in $F^l$ 
 so that,
 after some matrix manipulations and a change of variables, we get 
 \begin{IEEEeqnarray}{rl}\Lambda_{l+1}(W,a)=\frac{1}{|a_1||a_2|^2\cdots|a_l|^l}&\int_{\mathfrak{p}_{0,l}\backslash M_{l}}\int_{F^l}W\left[\sigma\left(\begin{array}{cccc}1_{l}&0&Z'&0\\0&1_{n-l}&0 &0\\ 0&0&1_{l}&0\\0&0&0&1_{n-l}\end{array}\right)\right.\nonumber \\
 &\times \left.\left(\begin{array}{cc} a &0\\0&a\end{array}\right)\left(\begin{array}{cccccc} 1_l&0&0&0&0&0\\0&1&0&Y&0&0\\0&0&1_{n-l-1}&0&0&0\\0&0&0&1_l&0&0\\0&0&0&0&1&0\\0&0&0&0&0&1_{n-l-1}\end{array}\right)\right]\nonumber \\
 & \times \psi(-\textrm{Tr}(Z'))dY\,dZ'.\nonumber\end{IEEEeqnarray}

In light of part (b) of Lemma \ref{dixmier1}, we may assume that $W$ is a finite sum of the form $$W(g)=\sum_j \int_{F^l}W_j\left[g\left(\begin{array}{cccccc} 1_l&0&0&0&0&0\\0&1&0&0&0&0\\0&0&1_{n-l-1}&0&0&0\\0&0&0&1_l&u&0\\0&0&0&0&1&0\\0&0&0&0&0&1_{n-l-1}\end{array}\right)\right]\Phi_j(u)du$$ where $\Phi_j$ is a Schwartz-Bruhat function on $F^l$.

Thus, we have 
\begin{IEEEeqnarray}{rCl}\Lambda_{l+1}(W,a)&=&\frac{1}{|a_1||a_2|^2\cdots|a_l|^l}\nonumber \\
&\times & \sum_j \int_{\mathfrak{p}_{0,l}\backslash M_{l}}\int_{F^l}\int_{F^l} W_j \left[\sigma\left(\begin{array}{cccccc} 1_l&0&0&0&Z'u^\prime&0\\0&1&0&0&Yu&0\\0&0&1_{n-l-1}&0&0&0\\0&0&0&1_l&u^\prime&0\\0&0&0&0&1&0\\0&0&0&0&0&1_{n-l-1}\end{array}\right)\right.\nonumber \\ 
&& \left.\left(\begin{array}{cccc}1_{l}&0&Z'&0\\0&1_{n-l}&0 &0\\ 0&0&1_{l}&0\\0&0&0&1_{n-l}\end{array}\right)\left(\begin{array}{cc} a &0\\0&a\end{array}\right)\left(\begin{array}{cccccc} 1_l&0&0&0&0&0\\0&1&0&Y&0&0\\0&0&1_{n-l-1}&0&0&0\\0&0&0&1_l&0&0\\0&0&0&0&1&0\\0&0&0&0&0&1_{n-l-1}\end{array}\right)\right]\nonumber \\ 
&\times&\Phi_j(u)du\,dY\,\psi(-\textrm{Tr}(Z'))dZ'\nonumber \end{IEEEeqnarray}
where $$u^\prime={}^t(a_1a_2\cdots a_lu_1,a_2\cdots a_l u_2, \ldots, a_lu_l).$$

Note that the conjugate under $\sigma$ of the first matrix on the left is in $N_{2n}$. 
Furthermore, $$\theta\left(\sigma\left(\begin{array}{cccccc} 1_l&0&0&0&Z'u^\prime&0\\0&1&0&0&Yu&0\\0&0&1_{n-l-1}&0&0&0\\0&0&0&1_l&u^\prime&0\\0&0&0&0&1&0\\0&0&0&0&0&1_{n-l-1}\end{array}\right)\sigma^{-1}\right)=\psi(Yu).$$ 

Using the transformation property of $W_j$ with respect to $\theta$, and setting 
$$\widehat{\Phi}(Y)=\int_{F^l}\Phi(u)\psi(Yu)du,$$ we have 
\begin{IEEEeqnarray}{rCl}\Lambda_{l+1}(W,a)&=&\frac{1}{|a_1||a_2|^2\cdots|a_l|^l}\sum_j \int_{\mathfrak{p}_{0,l}\backslash M_{l}}\int_{F^l}W_j\left[\sigma\left(\begin{array}{cccc}1_{l}&0&Z'&0\\0&1_{n-l}&0 &0\\ 0&0&1_{l}&0\\0&0&0&1_{n-l}\end{array}\right)\right.\nonumber \\
&& \times\left.\left(\begin{array}{cc} a &0\\0&a\end{array}\right)\left(\begin{array}{cccccc} 1_l&0&0&0&0&0\\0&1&0&Y&0&0\\0&0&1_{n-l-1}&0&0&0\\0&0&0&1_l&0&0\\0&0&0&0&1&0\\0&0&0&0&0&1_{n-l-1}\end{array}\right)\right]\widehat{\Phi_j}(Y)dY\,\psi(-\textrm{Tr}(Z'))dZ'.\nonumber \end{IEEEeqnarray}

Set $$W_j'(g)=\int_{F^l}W_j\left[g\left(\begin{array}{cccccc} 1_l&0&0&0&0&0\\0&1&0&Y&0&0\\0&0&1_{n-l-1}&0&0&0\\0&0&0&1_l&0&0\\0&0&0&0&1&0\\0&0&0&0&0&1_{n-l-1}\end{array}\right)\right]\widehat{\Phi_j}(Y)dY.$$
Note that $\widehat{\Phi_j}$ is a Schwartz-Bruhat function. In light of the proof of Lemma \ref{dixmier1}, $W_j'$ again defines an element of $\mathcal{W}(\pi,\psi)$.

Thus \begin{IEEEeqnarray}{rCl}\Lambda_{l+1}(W,a)&=&\frac{1}{|a_1||a_2|^2\cdots|a_l|^l}  \sum_j \int_{\mathfrak{p}_{0,l}\backslash M_{l}}W_j'\left[\sigma\left(\begin{array}{cccc}1_{l}&0&Z'&0\\0&1_{n-l}&0 &0\\ 0&0&1_{l}&0\\0&0&0&1_{n-l}\end{array}\right)\left(\begin{array}{cc} a &0\\0&a\end{array}\right)\right]\psi(-\textrm{Tr}(Z'))dZ',\nonumber \end{IEEEeqnarray}
or
$$\Lambda_{l+1}(W,a)=\sum_j\frac{1}{|a_1||a_2|^2\cdots|a_l|^l} \Lambda_{l}(W_j',a).$$

We see that the integral $\Lambda_{l+1}$ can be computed as a finite linear combination of the integrals $\Lambda_l$.

The assertion now follows by induction. Note that the fixed set of finite functions is changed by multiplying by the factor $$\frac{1}{|a_1||a_2|^2\cdots|a_l|^l}$$ but remains finite and independent of the choice of $W$.

Now we end the proof by noting, as in Section 3.11 and 4.4 of \cite{jasha},  that the Schwartz-Bruhat functions $\phi_\xi$ can be chosen to depend continuously on $W$, and as the procress of changing from $W$ to $\sum_j W_j'$ is a continuous map, we have the continuous dependence on $k$, as desired.

\end{proof}


We now consider integrals of the type $$\int_{(F^\times)^{n-1}} \phi(a_1,a_2,\ldots,a_{n-1})\mu_1(a_1)\mu_2(a_2)\cdots \mu_{n-1}(a_{n-1})d^\times a_1\, d^\times a_2\cdots d^\times a_{n-1}$$ where $\phi$ is a Schwartz-Bruhat function on $F^{n-1}$ and the $\mu_j$ are finite functions on $F^\times$ like those occurring in Lemma \ref{gauge}. This type of integral transform is sometimes called a generalized Mellin Transform of the function $\phi$. 
The analytic properties for such integrals in the one variable case are summarized in section 3.3 of \cite{jasha}.

We wish to generalize to the multivariable setting. 

\begin{prop}\label{expansion}
For each $j$, $1\leq j\leq n-1$, let $\mu_j$ denote a finite function on $F^\times$ of the form $$\mu_j(a)=\chi_0(a)^{m_j}|a|^{s+\lambda_j}(\log |a|)^{n_j}$$ where $s$ and $\lambda_j\in\mathbb{C}$, $n_j$ and $m_j$ are nonnegative integers and $\chi_0(a)=\frac {a}{|a|}$ if $F=\mathbb{R}$ or $\chi_0(a)=\frac {a}{|a|^{1/2}}$ if $F=\mathbb{C}$.

Let $\phi$ be a Schwartz-Bruhat function on $F^{n-1}$. Then the integral \begin{equation}F_\phi(s)=\int_{(F^\times)^{n-1}} \phi(a_1,a_2,\ldots,a_{n-1})\mu_1(a_1)\mu_2(a_2)\cdots \mu_{n-1}(a_{n-1})d^\times a_1 d^\times\, a_2\cdots d^\times a_{n-1}\label{genericgauge}\end{equation}
converges absolutely for $\Re(s)\gg 0$. Furthermore, it has a meromorphic continuation to the whole complex plane. In particular, if $M$ is any real number, then there exits a polynomial $Q_M(s)$, depending only on $M$ and $\mu_j$'s, such that, for each $\phi$, there exists a function $H_\phi(s)$ which is holomorphic in $s$ for $\Re(s)>-M$ such that, for $\Re(s)>-M$, \begin{equation}F_\phi (s)=\frac {H_\phi (s)}{Q_M (s)}.\label{gaugeexpansion}\end{equation}

Furthermore, $H_\phi$ depends continuously on $\phi$ in the following sense: For any compact subset $C$ of $\mathbb{C}$ contained in the right half plane $\Re(s)>-M$, there exists a continuous seminorm $p_C$ on the space of Schwartz-Bruhat functions on $F^{n-1}$, such that 
\begin{equation}|H_\phi(s)|\leq p_C(\phi)\label{continuity}\end{equation} for all $s\in C$.
\end{prop}

Remark: The continuity established in (\ref{continuity}) implies that, in our case, as $k\rightarrow k_0$ in $K_r$, then $\phi(\cdot,k)\rightarrow \phi(\cdot, k_0)$ in the space of Schwartz-Bruhat functions on $F^{n-1}$, hence $H_{\phi(\cdot,k)}(s)\rightarrow H_{\phi(\cdot,k_0)}(s)$ uniformly on compact sets.

\begin{proof}
That the integral on the right hand side of (\ref{genericgauge}) converges absolutely for $\Re(s)\gg 0$ follows from the fact that $\phi$ is a Schwartz-Bruhat function, and so decays rapidly and the fact that if $\Re(s)>-\Re(\lambda_j)$ for all $j$, the integral $$\int_{|a_j|<1}|a_j|^{\Re(s+\lambda_j)}(\log |a_j|)^{n_j} d^\times a_j$$ converges absolutely for all $j$.

Given a real number $M$, fix a positive integer $N$ such that $\min_j(\Re(\lambda_j))+(N+1)>M$.

For the rest of the proof, we introduce some notation. 

Let $$\mathbf{a}=(a_1,a_2,\ldots a_{n-1})\in (F^\times)^{n-1}.$$

For a multi-index $\alpha=(l_1,l_2,\ldots,l_{n-1})$ with the $l_j$  nonnegative integers, we will use the following standard notation. Let $$\alpha!=l_1!l_2!\cdots l_{n-1}!$$

For a smooth function $f:F^{n-1}\rightarrow \mathbb{C}$, set $$D^\alpha (f)=\frac {\partial^{l_{1}+l_{2}+\cdots+l_{{n-1}}} f}{\partial a_{1}^{l_{1}}\partial a_{2}^{l_{2}}\cdots \partial a_{{n-1}}^{l_{{n-1}}}}.$$

For a vector $\mathbf{x}=(x_1,x_2,\ldots, x_{n-1})\in F^{n-1}$, we set $$\mathbf{x}^\alpha=x_1^{l_1}x_2^{l_2}\cdots x_{n-1}^{l_{n-1}}.$$

We will also use the notation $${\left.f(\mathbf{a})\right|}_{a_j=x}=f(a_1,a_2,\ldots, a_{i-1}, x, a_{i+1},\ldots, a_{n-1}).$$ In some cases we will use this with $x$ a constant to denote fixing the variable $a_j$ in the function $f$, and letting the other variables continue to vary. In other cases we will use this to denote a change in variables, i.e., $\left.f(\mathbf{a})\right|_{a_j=a_j t}$ denotes replacing the variable $a_j$ with $a_j t$ in $f$.

Set $$\mu(\mathbf{a})=\mu_1(a_1)\mu_2(a_2)\cdots \mu_{n-1}(a_{n-1})$$
and $$d^\times \mathbf{a}=d^\times a_1 \,d^\times a_2\cdots d^\times a_{n-1}.$$

Let $S_{n-1}=\{1,2,\ldots, n-1\}.$ For each subset $T\subset S_{n-1}$ and each $j\in S_{n-1}$, define a subset $I_{j,T}$ of $F$ as follows: $$I_{j,T}=\left\{\begin{array}{ll} \{a:|a|<1\}& \textrm{if }j\in T\\ \{a:|a|\geq 1\}& \textrm{if }j\notin T.\end{array}\right.$$

For each subset $T$ of $S_{n-1}$, let $$I_T=I_{1,T}\times I_{2,T}\times \cdots \times I_{n-1,T}$$ which is a subset of $F^{n-1}$.

Note that $F^{n-1}$ is a disjoint union of the $I_T$ as $T$ ranges over all the subsets (including the empty set) of $S_{n-1}$. Thus, we can write $F_\phi(s)$ as a finite sum $$\sum_{T\subset S_{n-1}} \int_{I_{T}} \phi(\mathbf{a})\mu(\mathbf{a})d^\times \mathbf{a}.$$

Now fix a subset $T\subset S_{n-1}$.

Assume, for now, $F=\mathbb{R}$.

We expand $\phi$ using the Taylor expansion along with the integral form of the remainder \cite[Theorem 8.14]{fitz} for each variable $a_j$ with $j\in T$.

If $j\in T$, we may write \begin{equation}\phi(\mathbf{a})=\sum_{l_j=0}^N \frac 1{{l_j}!} a_j^{l_j}{\left.\frac{\partial^{l_j}\phi}{\partial a_j^{l_j}}(\mathbf{a})\right|}_{a_j=0}+ \frac {1}{N!} a_j^{N+1}\int_0^{1} {\left.\frac {\partial^{N+1} \phi}{\partial a_j^{N+1}}(\mathbf{a})\right|}_{a_j=a_j t_j}(1-t_j)^N dt_j.\label{onevariableexpansion}\end{equation}

Note that each of the partial derivatives appearing in the above expansion is again a Schwartz-Bruhat function in $n-2$ variables or, in the case of the last term, for $t$ fixed, of $n-1$ variables.
Thus, if $j'\neq j$ is also in $T$, we can expand each of the functions above in the variable $a_{j'}$, and we get 
\begin{IEEEeqnarray}{rCl}\phi(\mathbf{a})&=&\sum_{l_{j'}=0}^N \sum_{l_j=0}^N \frac 1{l_{j'}!{l_j}!}a_{j'}^{l_{j'}}a_j^{l_j}{\left.\frac{\partial^{l_{j'}}\partial^{l_j}\phi}{\partial a_{j'}^{l_{j'}}\partial a_j^{l_j}}(\mathbf{a})\right|}_{\substack{a_{j'}=0\\a_j=0}}\nonumber \\
&&+ \sum_{l_{j'}=0}^N \frac {1}{l_{j'}!N!} a_{j'}^{l_{j'}}a_j^{N+1}\int_0^{1} {\left.\frac {\partial^{l_{j'}}\partial^{N+1} \phi}{\partial a_{j'}^{l_{j'}}\partial a_j^{N+1}}(\mathbf{a})\right|}_{\substack{a_{j'}=0\\a_j=a_j t_j}}(1-t_j)^N dt_j\nonumber\\
&&+\sum_{l_j=0}^N \frac {1}{N!l_j!} a_{j'}^{N+1}a_j^{l_j}\int_0^{1} {\left.\frac {\partial^{N+1}\partial^{l_j} \phi}{\partial a_{j'}^{N+1}\partial a_j^{l_j}}(\mathbf{a})\right|}_{\substack{a_{j'}=a_{j'} t_{j'}\\a_j=0}}(1-t_{j'})^N dt_{j'}\nonumber \\
&&+\frac {1}{(N!)^2} a_{j'}^{N+1}a_j^{N+1}\int_0^{1}\int_0^1 {\left.\frac {\partial^{N+1}\partial^{N+1} \phi}{\partial a_{j'}^{N+1}\partial a_j^{N+1}}(\mathbf{a})\right|}_{\substack{a_{j'}=a_{j'} t_{j'}\\a_j=a_jt_j}}(1-t_{j'})^N(1-t_j)^N dt_j\, dt_{j'}.\nonumber \end{IEEEeqnarray}

Now, we continue this process for each variable $a_j$ with $j\in T$. We see that $\phi$ can be written as a finite sum as follows.

Let $T'$ denote any subset of $T$ (possibly the empty set), and $\alpha_{T,T'}$ denote a multi-index of the form $(l_1,l_2,\ldots, l_{n-1})$ with the $l_j$ nonnegative integers and $l_j=0$ if $j\notin T$, $0\leq l_j\leq N+1$ if $j\in T$ with the further requirement that $l_j=N+1$ if $j\in T'$.

We see that $\phi$ can be written as a finite sum (as $T'$ ranging over subsets of $T$) of functions of the form
\begin{equation}
\frac{(N+1)^{\# T'}}{\alpha_{T,T'}!}\mathbf{a}^{\alpha_{T,T'}}\int_{[0,1]^{\# T'}} D^{\alpha_{T,T'}}
\phi(\mathbf{u}_{T,T'}(\mathbf{a}))\prod_{j\in T'} (1-t_j)^N\prod_{j\in T'}dt_j
\label{expanded}\end{equation}
where $\mathbf{u}_{T,T'}(\mathbf{a})=(u_1,\ldots, u_{n-1})$ with
$$u_j=\left\{\begin{array}{ll}a_j& \textrm{if }j\notin T\\0& \textrm{if }j\in T\backslash T'\\ a_jt_j & \textrm{if }j\in T'.\end{array}\right.$$

Now consider each of the inegrals $$\int_{I_{T}} \phi(\mathbf{a})\mu(\mathbf{a})d^\times \mathbf{a}.$$

We replace $\phi$ with a sum of expressions of the form (\ref{expanded}) and get $F_\phi(s)$ as a sum of integrals of the form
$$\int_{I_{T}}\frac{(N+1)^{\# T'}}{\alpha_{T,T'}!}\mathbf{a}^{\alpha_{T,T'}}\int_{[0,1]^{\# T'}} D^{\alpha_{T,T'}}\phi(\mathbf{u}_{T,T'}(\mathbf{a}))\prod_{j\in T'} (1-t_j)^N\prod_{j\in T'}dt_j \,\mu(\mathbf{a})d^\times \mathbf{a}.$$
Note that, if $j\in T\backslash T'$, we are integrating the variable $a_j$ over the set $|a_j|<1$ and the expression $D^{\alpha_{T,T'}}\phi(\mathbf{u}_{T,T'})(\mathbf{a}))$ is constant with respect to $a_j$. 
Therefore, the contribution to the integral of the variable $a_j$ (for $j\in T\backslash T'$) can be factored out as $$\int_{|a_j|<1}a_j^{l_j}\mu_j(a_j)d^\times a_j=\left\{\begin{array}{ll}\frac {-2n_j!}{(s+\lambda_j+l_j)^{n_j}} &\textrm{if }m_j\textrm{ and }l_j\textrm{ have the same parity}\\ 0 &\textrm{otherwise,}\end{array}\right.$$ 
for $\Re(s)\gg0$.

Thus, for $\Re(s)\gg0$, $F_\phi(s)$ is a sum of expressions of the form 
\begin{IEEEeqnarray}{l}\frac{(N+1)^{\# T'}}{\alpha_{T,T'}!}\left(\prod_{j\in T\backslash T'} \frac{-2n_j!}{(s+\lambda_j+l_j)^{n_j}}\right)
 \int \prod_{j\in T'}(a_j^{N+1})\int_{[0,1]^{\# T'}} D^{\alpha_{T,T'}}\phi(\mathbf{u}_{T,T'}(\mathbf{a}))\nonumber \\
\times\prod_{j\in T'} (1-t_j)^N\prod_{j\in T'}dt_j \,\prod_{j\notin T\backslash T'}\mu_j(a_j)\prod_{j\notin T\backslash T'}d^\times a_j\label{hphi}\end{IEEEeqnarray}
where such a term only appears if $\alpha_{T,T'}$ is such that $l_j+m_j$ is even for each $j$.
Here, the integral in $a_j$ is taken over $|a_j|<1$ if $j\in T'$, and over $|a_j|\geq 1$ if $j\notin T$.

Set 
\begin{IEEEeqnarray}{rCl}H_{\phi,T,T'}(s)&=&\frac{(N+1)^{\# T'}}{\alpha_{T,T'}!}\int \prod_{j\in T'}(a_j^{N+1})\int_{[0,1]^{\# T'}} D^{\alpha_{T,T'}}\phi(\mathbf{u}_{T,T'}(\mathbf{a}))\nonumber \\ 
&& \times\prod_{j\in T'} (1-t_j)^N\prod_{j\in T'}dt_j \,\prod_{j\notin T\backslash T'}\mu_j(a_j)\prod_{j\notin T\backslash T'}d^\times a_j\label{hphireal}\end{IEEEeqnarray}

Now for any integer $L$, we define a seminorm $p_{N,L}$ on the space of Schwartz-Bruhat functions on $F^{n-1}$ 
as follows: for each multi-index $\alpha_{T,T'}$ as described above (there are only finitely many possibilities) set  $$p_{N,L,\alpha_{T,T'}}(\phi)=\max_{\mathbf{a}\in F^{n-1}}{\left( 
\left(\prod_{j\notin T}|a_j|^L\right)\left|D^\alpha_{T,T'}\phi(a_1,\ldots,a_{n-1})\right|
\right)}.$$
This is well-defined since $\phi$ and all of its partial derivatives are Schwartz-Bruhat functions.
Set $$p_{N,L,T,T'}(\phi)=\max_{\alpha_{T,T'}}(p_{N,L,\alpha_{T,T'}}(\phi)).$$
Finally, set $$p_{N,L}(\phi)= \max_{\substack{T,T'\\ T'\subset T\subset S_{n-1}}} (p_{N,L,T,T'}(\phi)).$$
Then $p_{N,L,\alpha_{T,T'}}$, $p_{N,L,T,T'}$ and $p_{N,L}$ all define continuous seminorms on the space of Schwartz-Bruhat functions.

Thus, for each $T\subset S_{n-1}$, $T'\subset T$ and $\alpha_{T,T'}$, we have $$\prod_{j \notin T} |a_j|^L\left|D^{\alpha_{T,T'}}\phi(\mathbf{a})\right|\leq p_{N,L}(\phi)$$
or $$\left|D^{\alpha_{T,T'}}\phi(a_1,\ldots, a_{n-1})\right|\leq \frac{p_{N,L}(\phi)}{\prod_{j\notin T}|a_j|^L}$$ for all $(a_1,\ldots,a_{n-1})$.
In particular $$\left|D^{\alpha_{T,T'}}\phi(\mathbf{u}_{T,T'}(\mathbf{a}))\right|\leq \frac{p_{N,L}(\phi)}{\prod_{j\notin T}|a_j|^L}$$ for all $\mathbf{a}\in F^{n-1}$ and $t_j\in [0,1]$ for all $j\in T'$.

Therefore the absolute value of the integral (\ref{hphireal}) is majorized by 
\begin{IEEEeqnarray}{Cll}&p_{N,L}(\phi)(N+1)^{\# T'}\int & \prod_{j\notin T} \left(|a_j|^{-L}\right)\prod_{j\in T'}\left(|a_j|^{N+1}\int_{[0,1]}  (1-t_j)^N dt_j\right)\prod_{j\notin T\backslash T'}|\mu_j(a_j)|\prod_{j\notin T\backslash T'}d^\times a_j \nonumber \\
=&\IEEEeqnarraymulticol{2}{l}{p_{N,L}(\phi)\int \left(|a_j|^{-L}\right)\prod_{j\in T'}\left(|a_j|^{N+1}\right) \prod_{j\notin T\backslash T'}|\mu_j(a_j)|\prod_{j\notin T\backslash T'}d^\times a_j.}\nonumber \end{IEEEeqnarray}

This is a constant times a product of integrals, each in a single variable $a_j$, with $j\notin T\backslash T'$. If $j\notin T$, we have the integral $$\int_{|a_j|\geq 1} |a_j|^{-L}|\mu_j(a_j)| d^\times a_j$$ which converges absolutely so long as $L>\Re(s+\lambda_j)$ or $\Re(s)<L-\Re(\lambda_j)$. Since $L$ is arbitrary, we can choose $L$ so that these integrals converge absolutely for all $s$ with $\Re(s)>-M$. 

On the other hand, if $j\in T'$, the integral is $$\int_{|a_j|<1} |a_j|^{N+1}|\mu_j(a_j)|d^\times a_j$$ which converges absolutely for $\Re(s)>-\Re(\lambda_j)-{N+1}$

Thus (\ref{hphireal}) converges absolutely for $\Re(s)>-\min_j(\Re(\lambda_j))-(N+1)$, and so $H_{\phi,T,T'}(s)$ defines a holomorphic function in $s$ for $\Re(s)>-M$ (recall that we chose $N$ so that $\min_j(\Re(\lambda_j))+(N+1)>M$).

We have shown that $F_\phi(s)$ is a sum of functions of the form \begin{equation}\frac{1}{\alpha_{T,T'}!}\left(\prod_{j\in T\backslash T'} \frac{-2n_j!}{(s+\lambda_j+l_j)^{n_j}}\right)H_{\phi,T,T'}(s)\label{meroreal}\end{equation} where $\alpha_{T,T'}$ is such that the $l_j$'s are integers with $0\leq l_j\leq N+1$ and $l_j+m_j$ is even for all $j$, and $H_{\phi,T,T'}(s)$ is holomorphic in $s$ for $\Re(s)>-M$. Set $Q_M(s)$ to be the least common denominator of the terms $\frac{1}{(s+\lambda_j+l_j)^{n_j}}$ in this sum, and we get $$F_\phi(s)=\frac{H_\phi(s)}{Q_M(s)},$$ as desired, where $H_\phi(s)$ is the sum $$\sum_{T'\subset T\subset S_{n-1},\,\alpha_{T,T'}'} Q_M(s)\left(\prod_{j\in T}\frac{1}{\alpha_{T,T'}'!}\prod_{j\in T\backslash T'} \frac{-2n_j!}{(s+\lambda_j+l_j)^{n_j}}\right)H_{\phi,T,T'}(s)$$ where the sum is taken as $T$ ranges over all the subsets of $S_{n-1}$ and $T'$ ranges over all the subsets of $T$, and all $\alpha_{T,T'}'$ all multi-indices with $l_j=0$ if $j\notin T$, $l_j=N+1$ if $t\in T'$ and $0\leq l_j \leq N+1$ with $l_j+m_j$ even for all $j$.
$H_\phi(s)$ is then meromorphic for $\Re(s)>-M$.

Since $M$ is arbitrary, this gives a meromorphic continuation of (\ref{genericgauge}) to the entire complex plane.

To establish the desired continuity, we take any compact set $C$ in the right half plane $\Re(s)>-M$, and we choose $L$ above so that $C$ is contained in the strip $-M<\Re(s)<L-\max_j(\Re(\lambda_j))$. Then $|H_\phi(s)|$ is majorized by the finite sum 
\begin{IEEEeqnarray}{Cll}&\IEEEeqnarraymulticol{2}{l}{\sum_{\substack{T'\subset T\subset S_{n-1}\\ \alpha_{T,T'}'}}\left| Q_M(s)\left(\prod_{j\in T}\frac{1}{l_j!}\prod_{j\in T\backslash T'} \frac{-2n_j!}{(s+\lambda_j+l_j)^{n_j}}\right)\right||H_{\phi,T,T'}(s)|}\nonumber \\
\leq & p_{N,L}(\phi) \sum_{\substack{T'\subset T\subset S_{n-1}\\ \alpha_{T,T'}'}}
& \left| Q_M(s)\left(\prod_{j\in T}\frac{1}{l_j!}\prod_{j\in T\backslash T'} \frac{-2n_j!}{(s+\lambda_j+l_j)^{n_j}}\right)\right|\nonumber \\
&&\times \int \left(|a_j|^{-L}\right)\prod_{j\in T'}\left(|a_j|^{N+1}\right) \prod_{j\notin T\backslash T'}|\mu_j(a_j)|\prod_{j\notin T\backslash T'}d^\times a_j.\nonumber \end{IEEEeqnarray}

The expression in the sum defines a continuous function in $s$ for $-M<\Re(s)<L-\max_j(\Re(\lambda_j))$, and so has a maximum value for $s\in C$. Take $p_C$ to be $p_{N,L}$ multiplied by this maximum value and we have the desired result.

For the complex case (i.e. $F=\mathbb{C}$), in place of (\ref{onevariableexpansion}),
we use the Taylor Expansion of $\phi$ in each variable with respect to the Wirtinger Derivatives, which can be obtained using standard results from calculus of two real variables. See Appendix A.
\begin{IEEEeqnarray}{rCl}\phi(\mathbf{a})&=&\sum_{l_j=0}^N\sum_{k_j=0}^{l_j} \frac {1}{k_j!(l_j-k_j)!}a_j^{k_j}\overline{a_j}^{l_j-k_j}{\left. \frac {\partial^l_j \phi}{\partial a_j^{k_j} \partial \overline{a_j}^{l_j-k_j}}(\mathbf{a})\right|}_{a_j=0}\nonumber \\
&&+\sum_{k_j=0}^{N+1}\frac {N+1}{k_j!(N+1-k_j)!}a_j^{k_j}\overline{a_j}^{N+1-k_j}\int_0^1 {\left.\frac {\partial ^{N+1}\phi}{\partial a_j^{k_j}\partial \overline{a_j}^{N+1-k_j}}(\mathbf{a})\right|}_{a_j=a_jt}(1-t)^Ndt.\IEEEeqnarraynumspace\label{onevariablecomplex}\end{IEEEeqnarray}

We now proceed in the same fashion as in the real case, expanding $\phi$ in the same way in each variable $a_j$ with $j\in T$.

We adopt the convention that, given two multi-indices $\beta=(b_1,b_2,\ldots,b_{n-1})$ and $\gamma=(c_1,c_2,\ldots, c_{n-1})$ then $$\beta+\gamma=(b_1+c_1,b_2+c_2,\ldots, b_{n-1}+c_{n-1}).$$

Given a pair of multi-indices $\beta$ and $\gamma$ as above, and a smooth function $f$ on $n-1$ complex variables, we write  $$D^{\beta,\gamma} (f)=\frac {\partial^{b_{1}+b_{2}+\cdots+b_{n-1}+c_1+c_2+\cdots+c_{n-1}} f}{\partial a_{1}^{b_{1}}\partial a_{2}^{b_{2}}\cdots \partial a_{{n-1}}^{b_{{n-1}}}\partial {\overline{a_{1}}}^{c_{1}}\partial {\overline{a_{2}}}^{c_{2}}\cdots \partial {\overline{a_{{n-1}}}}^{c_{{n-1}}}}$$

Instead of (\ref{expanded}), we see $\phi$ can be written as finite sum of functions of the form 
\begin{equation}
\sum_{\beta+\gamma=\alpha_{T,T'}}\frac{(N+1)^{\# T'}}{\beta!\gamma!}\mathbf{a}^{\beta}\overline{\mathbf{a}}^\gamma\int_{[0,1]^{\# T'}} D^{\beta,\gamma}\phi(\mathbf{u}_{T,T'}(\mathbf{a}))\prod_{j\in T'} (1-t_j)^N\prod_{j\in T'}dt_j
\label{expandedcomplex}\end{equation}
where $\mathbf{u}_{T,T'}(\mathbf{a})=(u_1,\ldots, u_{n-1})$ with
$$u_j=\left\{\begin{array}{ll}a_j& \textrm{if }j\notin T\\0& \textrm{if }j\in T\backslash T'\\ a_jt_j & \textrm{if }j\in T'.\end{array}\right.$$
The integral is to be computed as an iterated line integral with $t_j$ real.

Thus $F_\phi(s)$ can be written as a finite sum (over $T$, $T'\subset T$ and $\alpha_{T,T'}$) of expressions of the form 
$$\sum_{\beta+\gamma=\alpha_{T,T'}} \int_{I_{T}}\frac{(N+1)^{\# T'}}{\beta!\gamma!}a^\beta \overline{a}^\gamma \int_{[0,1]^{\# T'}} D^{\beta,\gamma}\phi(\mathbf{u}_{T,T'}(\mathbf{a}))\prod_{j\in T'} (1-t_j)^N\prod_{j\in T'}dt_j \,\mu(\mathbf{a})d^\times \mathbf{a}.$$

As in the real case, for those $j\in T\backslash T'$, the expression $D^{\alpha, \beta}\phi(\mathbf{u}_{T,T'})(\mathbf{a}))$ is constant with respect to $a_j$ 
and the contribution to the integral of the variable $a_j$ (for $j\in T\backslash T'$) can be factored out as a factor of $$\int_{|a_j|<1}a_j^{b_j}\overline{a_j}^{c_j}\mu_j(a_j)d^\times a_j.$$

Converting to polar coordinates $a_j=r_je^{i\theta_j}$ and using the definition of $m_j$ and $\chi_0$, we see that this integral becomes $$\frac{1}{2\pi}\int_0^1 r^{2s+2\lambda_j+b_j+c_j} (\log r)^{n_j}d^\times r \int_0^{2\pi} e^{i(m_j+b_j-c_j)\theta} d\theta$$ which is 0 unless $m_j+b_j-c_j=0$. If this is the case, then it becomes $$\frac {-n_j!}{(2s+2\lambda_j+b_j+c_j)^{n_j}}.$$

Thus, for $\Re(s)\gg0$, $F_\phi(s)$ is a sum of expressions of the form 
\begin{IEEEeqnarray}{l}\frac{(N+1)^{\# T'}}{\beta!\gamma!}\left(\prod_{j\in T\backslash T'} \frac{-n_j!}{(2s+2\lambda_j+b_j+c_j)^{n_j}}\right) \int \prod_{j\in T'}a_j^{b_j}\overline{a_j}^{c_j}\nonumber\\
\times \int_{[0,1]^{\# T'}} D^{\beta, \gamma}\phi(\mathbf{u}_{T,T'}(\mathbf{a}))\prod_{j\in T'} (1-t_j)^N\prod_{j\in T'}dt_j \,\prod_{j\notin T\backslash T'}\mu_j(a_j)\prod_{j\notin T\backslash T'}d^\times a_j\label{hphicomplex}\end{IEEEeqnarray}
where such a term only appears if $\beta+\gamma=\alpha_{T,T'}$ with $b_j-c_j=m_j$ for each $j\in T\backslash T'$. Note that $\beta+\gamma=\alpha_{T,T'}$ ensures that $b_j+c_j=N+1$, so that $|a_j|^{b_j}|\overline{a_j}|^{c_j}=|a_j|^{N+1}$.

Furthermore, the functions $D^{\beta, \gamma} \phi$ can be written in terms of the partial derivatives of $\phi$ with respect to $x_j$ and $y_j$, so they are again Schwartz-Bruhat functions.

From here, we may proceed as in the real case to obtain the desired results. Note that we need to take $N$ such that $\min_j(\Re(\lambda_j))+\frac{N+1}{2}>M$ in this case.

\end{proof}

\subsection{The Space $C_c^\infty(\theta_{m},GL_{m})$}

Let $m$ be any positive integer. For $F=\mathbb{R}$  or $\mathbb{C}$, there is another formulation for the set $C_c^\infty(\theta_m, GL_{m}).$

Consider the matrix $$g=\left[\begin{array}{c} v_1\\v_2\\\vdots \\v_m\end{array}\right]\in GL_m$$ with $v_i\in F^m$ as row vectors. Let $\langle\cdot,\cdot\rangle$ denote the usual inner product on $F^m$. We look at the result of the Gram-Schmidt process, starting with $v_m$ and working inductively up, $$GS(g)=\left[\begin{array}{c} w_1\\w_2\\\vdots \\ w_m\end{array}\right],$$ 
\begin{IEEEeqnarray}{rCl}w_m &=&v_m\nonumber \\ w_{m-1}& =&v_{m-1}-\frac{\langle v_{m-1},w_m\rangle}{\langle w_m,w_m\rangle} w_m\nonumber \\ &\vdots& \nonumber \\ w_1&=&v_1-\frac{\langle v_1,w_2\rangle}{\langle w_2,w_2\rangle}w_2-\cdots-\frac{\langle v_1,w_m\rangle}{\langle w_m,w_m\rangle}w_m.\nonumber\end{IEEEeqnarray}

The observant reader will notice that this is part of the development of the Iwasawa decomposition for $GL_m(F)$. For our purposes, it is useful to look at the construction directly, as we wish to make the following observations and definitions.

Note that the map $g\mapsto w_i$ is smooth for all $i$. Furthermore, note that this process is invariant under left translation by elements of $N_m$. That is $$GS(ug)=GS(g)$$ for all $u\in N_m$.

For $1\leq i\leq m-1$ define $\lambda_i:GL_m\rightarrow \mathbb{R}$ by $$\lambda_i(g)=\frac{\langle v_i,w_{i+1}\rangle}{\langle w_{i+1},w_{i+1}\rangle}.$$ Then $\lambda_i$ is smooth.

Furthermore, for $$u=\left(\begin{array}{ccc}\begin{array}{cc}1&u_{1,2}\\0&1\end{array}&\cdots&\begin{array}{cc} *&* \\ *&* \end{array}\\ \vdots&\ddots &\vdots\\ \begin{array}{cc}0&0\\0&0\end{array}& \cdots &\begin{array}{cc}1&u_{m-1,m}\\0&1\end{array}\end{array}\right)\in N_m,$$
$$\lambda_i(ug)=u_{i,i+1}+\lambda_i(g).$$

Define $\Theta_m:GL_m\rightarrow \mathbb{C}$ by $$\Theta_m(g)=\prod_{i=1}^{m-1}\psi(\lambda_i(g)).$$
Then $\Theta_m$ is smooth, $|\Theta_m(g)|=1$ for all $g$, and for $u\in N_m$, $$\Theta_m(ug)=\theta_m(u)\Theta_m(g)$$
for all $g\in GL_m$, and $$\Theta_m(u)=\theta_m(u).$$

Note that if $GL_m(F)=N_mA_mK_m$ is the Iwasawa decomposition of $GL_m$, then $\Theta_m(uak) = \theta_m(u)$ for $u\in N_m$, $a\in A_m$ and $k\in K_m$.

Furthermore, every $\phi\in C_c^\infty(\theta_m,GL_m)$ can be written as $$\phi(g)=\Theta_m(g)f([g])$$ where $$f\in C_c^\infty (N_m\backslash GL_m)$$ and $[g]$ denotes the class of $g$ in $N_m\backslash GL_m.$

Let $\sigma'$ denote the $(r-1) \times (r-1)$ matrix defined by the upper left block of the matrix $\sigma$, i.e. $$\sigma=\left(\begin{array}{cc}\sigma'&0\\0&1\end{array}\right).$$

We will require the following result.

\begin{lem}\label{existence} Let $r=2n$ be even and $\sigma'\in GL_{r-1}$ defined as above. Let $H$ denote the closed subgroup of $GL_{r-1}$ consisting of matrices of the form $$\sigma'\left(\begin{array}{ccc}g&0&Z'\\0&1&Y\\0&0&g\end{array}\right)\sigma'^{-1}$$ where $g\in GL_{n-1}$, $Z'\in M_{n-1}$ and $Y\in F^{n-1}$ as a row vector. Then the homogeneous space $$(N_{r-1}\cap H)\backslash H$$ is a closed regular submanifold of the homogeneous space $$N_{r-1}\backslash GL_{r-1}.$$ Hence every function in $C_c^\infty((N_{r-1}\cap H)\backslash H)$ can be realized as the restriction of a function in $C_c^\infty (N_{r-1}\backslash GL_{r-1})$. \end{lem}

\begin{proof}

The natural injection defines $(N_{r-1}\cap H)\backslash H$ as a subset of $N_{r-1}\backslash GL_{r-1}$. We need to check that the differential structure endowed by $(N_{r-1}\cap H)\backslash H$ as a subspace of $N_{r-1}\backslash GL_{r-1}$ is the same as the original differential structure given to $(N_{r-1}\cap H)\backslash H$ as a quotient space of Lie groups. To do this, we examine the tangent spaces involved.

The tangent space, $T_e(N_{r-1}\backslash GL_{r-1})$, of $N_{r-1}\backslash GL_{r-1}$ can be realized as the vector space quotient of Lie Algebras $\frac{\textrm{Lie}(GL_{r-1})}{\textrm{Lie}(N_{r-1})}.$ But $\textrm{Lie}(GL_{r-1})$ is simply $M_{r-1}$, while $\textrm{Lie}(N_{r-1})$ is the subspace of $M_{r-1}$ consisting of upper triangular matrices with zeros along the diagonal. Thus the tangent space $T_e(N_{r-1}\backslash GL_{r-1})$ can be realized as the vector space consisting of $(r-1) \times (r-1)$ lower triangular matrices.

On the other hand $T_e((N_{r-1}\cap H)\backslash H)$ can also be realized as the vector space quotient of Lie Algebras $\frac {\textrm{Lie}(H)}{\textrm{Lie}(N_{r-1}\cap H)}.$ 

Let $L$ denote the group of matrices of the form $\left(\begin{array}{ccc}g&0&Z'\\0&1&Y\\0&0&g\end{array}\right).$ Then $\textrm{Lie}(L)$ consists of $(r-1) \times (r-1)$ matrices of the form $\left[\begin{array}{ccc} x&0&z\\0&0&y\\0&0&x\end{array}\right]$ where $x$ and $z$ are in $M_{n-1}$ and $y\in F^{n-1}$.  

Let $L'$ denote the subgroup of $L$ consisting of matrices of the form $\left(\begin{array}{ccc}u&0&T\\0&1&0\\0&0&u\end{array}\right)$ where $u\in N_{n-1}$ and $T\in \mathfrak{p}_{0,n-1}$ (i.e. $T$ is an upper triangular matrix). Then $\textrm{Lie}(L')$ consists of $(r-1) \times (r-1)$ matrices of the form  $\left[\begin{array}{ccc} \mathfrak{u}&0&\mathfrak{t}\\0&0&0\\0&0&\mathfrak{u}\end{array}\right]$ where $\mathfrak{u}$ is an upper triangular matrix with zeros along the diagonal, and $\mathfrak{t}$ is an upper triangular matrix.

Thus, the tangent space $T_e(L'\backslash L)$ is the vector space quotient $\frac {\textrm{Lie}(L)}{\textrm{Lie}(L')},$ which consists of those $(r-1) \times (r-1)$ matrices of the form $\left[\begin{array}{ccc} \mathfrak{n}&0&\mathfrak{z}\\0&0&\mathfrak{y}\\0&0&\mathfrak{n}\end{array}\right]$ where $\mathfrak{n}$ is a lower triangular matrix in $M_{n-1}$ and $\mathfrak{z}$ is a lower triangular matrix with zeros along the diagonal, and $\mathfrak{y}\in F^{n-1}$.

Now, we notice that $\sigma'L\sigma'^{-1}=H$ and $\sigma'L'\sigma'^{-1}=N_{r-1}\cap H.$ Thus $$T_e((N_{r-1}\cap H)\backslash H)=
\textrm{Ad}(\sigma')\frac {\textrm{Lie}(L)}{\textrm{Lie}(L')},$$ which consists of matrices of the form $$\textrm{Ad}(\sigma')\left[\begin{array}{ccc} \mathfrak{n}&0&\mathfrak{z}\\0&0&\mathfrak{y}\\0&0&\mathfrak{n}\end{array}\right]$$ with $\mathfrak{n}$ a lower triangular matrix, $\mathfrak{z}$ a lower triangular matrix with zeros along the diagonal, and $\mathfrak{y}\in F^{n-1}$. A straightforward computation reveals that matrices of this form are in fact lower triangular matrices. 

Thus we have a natural way of imbedding $T_e((N_{r-1}\cap H)\backslash H)$ into $T_e(N_{r-1}\backslash GL_{r-1})$ which shows that $(N_{r-1}\cap H)\backslash H$ has the desired closed regular submanifold structure.

The final statement then follows from well-known results from Differential Geometry (see, for example, chapter 5 of \cite{boothby}).

\end{proof}

\subsection{Proof of Lemma \ref{J1} in the Archimedean Case}

We will now prove Lemma \ref{J1} in the archimedean case.

\begin{proof}[Proof of Lemma \ref{J1} with $F=\mathbb{R}$ or $\mathbb{C}$.]

It is convenient here to replace $W$ in (\ref{j1}) with $\rho(\sigma^{-1})W$, where $\rho$ is the right regular representation, so that (\ref{j1}) becomes
\begin{multline}J_1(s,h,\chi,\rho(\sigma^{-1})W)=\chi(\det h)|\det h|^{s-1}\\
\times \int_{N_{n-1}\backslash G_{n-1}}\int_{\mathfrak{p}_{0,n}\backslash M_n}W\left[\sigma\left(\begin{array}{cc}1_n&Z\\0&1_n\end{array}\right)\left(\begin{array}{cccc} g &0&0&0\\0&1&0&0\\0&0&g&0\\0&0&0&1\end{array}\right)\left(\begin{array}{cc} h &0\\0&h\end{array}\right)\sigma^{-1}\right]\\
 \times\psi(-\textrm{Tr}(Z))dZ\,\chi(\det g)|\det g|^{s-1}dg\label{j1s}\end{multline}

We decompose $M_n(F)$ as $M_{n-1}(F)\oplus F^{n-1} \oplus F^{n-1}\oplus F$, writing $Z=\left(\begin{array}{cc} Z'&u\\Y& x\end{array}\right)$
with $Z'\in M_{n-1}$, $Y\in F^{n-1}$ as a row vector, $u\in F^{n-1}$ as a column vector and $x\in F$. Under this decomposition, the subspace $\mathfrak{p}_{0,n}$ decomposes as $\mathfrak{p}_{0,n-1}\oplus \{0\}\oplus F^{n-1} \oplus F$. Thus, the quotient space $\mathfrak{p}_{0,n}\backslash M_n$ decomposes as $\mathfrak{p}_{0,n-1}\backslash M_{n-1}\oplus F^{n-1}\oplus \{0\} \oplus \{0\}$. So, we write $Z=\left(\begin{array}{cc} Z'&0\\Y&0\end{array}\right)$ and as $Z$ ranges over $\mathfrak{p}_{0,n}\backslash M_n$, $Z'$ ranges over $\mathfrak{p}_{0,n-1}\backslash M_{n-1}$ and $Y$ ranges over $F^{n-1}$.

Also write $\sigma=\left(\begin{array}{cc}\sigma'&0\\0&1\end{array}\right)$ with $\sigma'\in GL_{r-1}$.

Thus, by Proposition \ref{smoothvector}, for $\phi_0\in C_c^\infty(\theta_{r-1},GL_{r-1})$, there exists a $W_{\phi_0}\in \mathcal{W}(\pi,\psi)$ such that 
$$W_{\phi_0}\left[\left(\begin{array}{cc}g&0\\0&1\end{array}\right)\right]=\phi_0(g).$$ 

Notice that $$\sigma\left(\begin{array}{cc}1_n&Z\\0&1_n\end{array}\right)\left(\begin{array}{cccc} g &0&0&0\\0&1&0&0\\0&0&g&0\\0&0&0&1\end{array}\right)\sigma^{-1}
=\left(\begin{array}{cc}\sigma'\left(\begin{array}{ccc}g&0&Z'g\\0&1&Yg\\0&0&g\end{array}\right)\sigma'^{-1} &0\\0&1\end{array}\right)$$
which is in $GL_{r-1}$ (as a subgroup of $GL_{r}$).

Taking $W=W_{\phi_0}$ and $h=1_n$ in (\ref{j1s}), we have
\begin{IEEEeqnarray}{rCl}J_1(s,1_n,\chi,\rho(\sigma^{-1})W_{\phi_0})& =&\int_{N_{n-1}\backslash G_{n-1}}\int_{F^{n-1}}\int_{\mathfrak{p}_{0,n-1}\backslash M_{n-1}}\phi_0\left[\sigma'\left(\begin{array}{ccc}g&0&Z'g\\0&1&Yg\\0&0&g\end{array}\right)\sigma'^{-1}\right]\nonumber \\
& &\times\psi(-\textrm{Tr}(Z'))dZ'\,dY\,
\chi(\det g)|\det g|^{s-1}dg.\label{ourcase}\end{IEEEeqnarray}

Let $H$ be the closed subgroup of $GL_{r-1}$ defined in Lemma \ref{existence}. Note that $H$ is the image of the continuous injection $$(g,Z',Y)\mapsto\sigma'\left(\begin{array}{ccc}g&0&Z'g\\0&1&Yg\\0&0&g\end{array}\right)\sigma'^{-1}$$
from $GL_{n-1}\times M_{n-1}\times F^{n-1}$ into $GL_{r-1}$. Note that any $Z''\in M_{n-1}$ can be written as $Z'g$ for some $Z'\in M_{n-1}$. Similarly for $Y'\in F^{n-1}$.

Notice also that $N_{r-1}\cap H$ is the image of $N_{n-1}\times\mathfrak{p}_{0,n-1}\times \{0\}$. In this sense, $N_{r-1}\backslash H$ can be identified with $N_{n-1}\backslash GL_{n-1}\times \mathfrak{p}_{0,n-1}\backslash M_{n-1}\times F^{n-1}$, and the iterated integral is an integral over the space $(N_{r-1}\cap H)\backslash H$. By Lemma \ref{existence}, the homogeneous space $(N_{r-1}\cap H)\backslash H$ is a closed regular submanifold of $N_{r-1}\backslash GL_{r-1}$. Thus, if $\phi_0$ is of compact support modulo $N_{r-1}$, then ${\phi_0|}_H$ is of compact support modulo $N_{r-1}\cap H$. 
Thus, the integral (\ref{ourcase}) is absolutely convergent for all values of $s$. This proves the first part of the lemma for any $W$ of the form $\rho(\sigma^{-1})W_{\phi_0}$.

Now fix $s=s_0$.

For $h\in H$, $$h=\sigma'\left(\begin{array}{ccc}g&0&Z'g\\0&1&Yg\\0&0&g\end{array}\right)\sigma'^{-1},$$ set 
$$A(h)=\Theta_{r-1}(h)\psi(-\textrm{Tr}(Z'))
\chi(\det g)|\det g|^{s_0-1},$$ where $\Theta_{r-1}$ is as defined earlier in this section (taking $m=r-1$).

Recall that, if $u\in N_{r-1}\cap H$, $u$ has the form $$u=\sigma'\left(\begin{array}{ccc} u_0&0&Tu_0\\0&1&0\\0&0&u_0\end{array}\right)\sigma'{-1}$$  with $u_0\in N_{n-1}$ and $T$ an  upper triangular matrix and $\theta_{r-1}(u)=\psi(\textrm{Tr}(T))$.
Then, for $u\in N_{r-1}\cap H$ and $h\in H$, $$A(uh)=A(h).$$ Thus $A$ defines a smooth function on the homogeneous space $(N_{r-1}\cap H)\backslash H$.

Notice that $A(1_{r-1})=\Theta_{r-1}(1_{r-1})\psi(0)\chi(\det 1_{r-1})|\det 1_{r-1}|^{s_0-1}=1$.

In light of Proposition \ref{existence}, the restrictions of smooth functions on $N_{r-1}\backslash GL_{r-1}$ to $(N_{r-1}\cap H)\backslash H$ are arbitrary among smooth functions on $(N_{r-1}\cap H)\backslash H$.
Thus, since $A(1_{r-1})=1$, we may choose $f\in C^\infty (GL_{r-1})$, invariant under left multiplication by elements of $N_{r-1}$ and of compact support modulo that subgroup (i.e. $f(g)=f'([g])$ for $f'\in C_c^\infty (N_{r-1}\backslash GL_{r-1})$) with support concentrated (modulo $N_{r-1}$) near $1_{r-1}$
such that 
\begin{IEEEeqnarray}{rCl}\IEEEeqnarraymulticol{3}{l}{\int_{(N_{r-1}\cap H)\backslash H} f(h)A(h)dh} \nonumber \\&= &\int_{N_{n-1}\backslash G_{n-1}}\int_{F^{n-1}}\int_{\mathfrak{p}_{0,n-1}\backslash M_{n-1}}f\left(\sigma'\left(\begin{array}{ccc}g&0&Z'g\\0&1&Yg\\0&0&g\end{array}\right)\sigma'^{-1}\right)\nonumber \\
& & \times \Theta_{r-1}\left(\sigma'\left(\begin{array}{ccc}g&0&Z'g\\0&1&Yg\\0&0&g\end{array}\right)\sigma'^{-1}\right)\psi(-\textrm{Tr}(Z'))dZ'\,dY\,\chi(\det g)|\det g|^{s_0-1}dg\nonumber \end{IEEEeqnarray}
is non-zero.

Since $f\Theta_{r-1}$ defines an element in $C_c^\infty(\theta_{r-1},GL_{r-1})$, we have proved the final statement in the lemma for $W=W_{\phi_0}$ where $\phi_0=f\Theta_{r-1}$.

\end{proof}

\subsection{A Meromorphic Continuation for $J_1(s,h,\chi,W)$}

We still assume $F=\mathbb{R}$ or $\mathbb{C}$.

Recall \begin{IEEEeqnarray}{rCl}J_1(s,h,\chi,W)&=& 
\chi(\det h)|\det h|^{s-1} \nonumber \\
&&\times \int_{N_{n-1}\backslash GL_{n-1}}\int_{\mathfrak{p}_{0,n}\backslash M_n}W\left[\sigma\left(\begin{array}{cc}1_n&Z\\0&1_n\end{array}\right)\left(\begin{array}{cccc} g &0&0&0\\0&1&0&0\\0&0&g&0\\0&0&0&1\end{array}\right)\left(\begin{array}{cc} h &0\\0&h\end{array}\right)\right]\nonumber \\
&&\times \psi(-\textrm{Tr}(Z))dZ\,\chi(\det g)|\det g|^{s-1}dg.\label{J1again}\end{IEEEeqnarray}

For now, we will assume that $h=k\in {K_n}$. Then we consider the integral 
\begin{IEEEeqnarray}{rCl}J_1(s,k,\chi,W)&=&\int_{N_{n-1}\backslash G_{n-1}}\int_{\mathfrak{p}_{0,n}\backslash M_n}W\left[\sigma\left(\begin{array}{cc}1_n&Z\\0&1_n\end{array}\right)\left(\begin{array}{cccc} g &0&0&0\\0&1&0&0\\0&0&g&0\\0&0&0&1\end{array}\right)\left(\begin{array}{cc} k &0\\0&k\end{array}\right)\right]\nonumber \\
&&\times \psi(-\textrm{Tr}(Z))dZ\,\chi(\det g)|\det g|^{s-1}dg.\label{justk}\end{IEEEeqnarray}

As the integrand $$\int_{\mathfrak{p}_{0,n}\backslash M_n}W\left[\sigma\left(\begin{array}{cc}1_n&Z\\0&1_n\end{array}\right)\left(\begin{array}{cccc} g &0&0&0\\0&1&0&0\\0&0&g&0\\0&0&0&1\end{array}\right)\left(\begin{array}{cc} k &0\\0&k\end{array}\right)\right]\psi(-\textrm{Tr}(Z))dZ$$ is invariant in $g$ under left translation by elements of $N_{n-1}$, we appeal to the Iwasawa decomposition of $g=nac$, $n\in N_{n-1}$, $a\in A_{n-1}$ and $c\in K_{n-1}$, and see that this integral is equal to the integral
\begin{IEEEeqnarray}{rCl}J_1(s,k,\chi,W)&=&\int_{K_{n-1}}\int_{(F^\times)^{n-1}}\int_{\mathfrak{p}_{0,n}\backslash M_n}\nonumber \\
&&W\left[\sigma\left(\begin{array}{cc}1_n&Z\\0&1_n\end{array}\right)\left(\begin{array}{cccc} a &0&0&0\\0&1&0&0\\0&0&a&0\\0&0&0&1\end{array}\right)\left(\begin{array}{cccc} c &0&0&0\\0&1&0&0\\0&0&c&0\\0&0&0&1\end{array}\right)\left(\begin{array}{cc} k &0\\0&k\end{array}\right)\right]\nonumber \\
&& \times \psi(-\textrm{Tr}(Z))dZ\,\chi(\det a)\delta_{n-1}(a)^{-1}|\det a|^{s-1}d^\times a\,dc.\label{justak}\end{IEEEeqnarray} 
where $$a=\textrm{diag}(a_1,a_2,\ldots,a_{n-1},1)$$
and $\delta_{n-1}$ is the modular character associated to the standard Borel subgroup of $GL_{n-1}$.

For ease of notation, we set $$\kappa(c,k)=\left(\begin{array}{cccc} c &0&0&0\\0&1&0&0\\0&0&c&0\\0&0&0&1\end{array}\right)\left(\begin{array}{cc} k &0\\0&k\end{array}\right)\in K_r.$$

We have $J_1(s,k,\chi,W)=$ $$\int_{K_{n-1}}\int_{(F^\times)^{n-1}}\Lambda_{n}\left(W,\left(\begin{array}{cccc} a &0&0&0\\0&1&0&0\\0&0&a&0\\0&0&0&1\end{array}\right)\kappa(c,k)\right)\chi(\det a)\delta_{n-1}(a)^{-1}|\det a|^{s-1}d^\times a\,dc.$$

Now, we apply Lemma \ref{gauge}, and replace $\Lambda_{n}$ with a finite sum of the form $$\sum \phi_\xi\left(\frac{a_1}{a_2},\frac {a_2}{a_3},\ldots,a_{n-1},\kappa(c,k)\right)\xi\left(\frac{a_1}{a_2},\frac {a_2}{a_3},\ldots,a_{n-1}\right)$$ where $\phi_\xi\in \mathcal{S}(F^{n-1}\times K_r)$ and $\xi$ are finite functions from the set $X$ described in Lemma \ref{gauge}.

Thus,  $J_1(s,k,\chi,W)$ is a finite sum of integrals of the form 
\begin{multline}\int_{K_{n-1}}\int_{(F^\times)^{n-1}}\phi_\xi\left(\frac{a_1}{a_2},\frac {a_2}{a_3},\ldots,a_{n-1},\kappa(c,k)\right)\mu_1\left(\frac{a_1}{a_2}\right)\mu_2\left(\frac{a_2}{a_3}\right)\cdots \mu_{n-1}(a_{n-1}) d^\times a_1d^\times a_2\cdots d^\times a_{n-1}\,dc\nonumber \end{multline}
where $\phi_\xi \in\mathcal{S}(F^{n-1}\times K_r)$, $$\mu_j(a)=\chi_0(a)^{m_j}|a|^{s+\lambda_j}(\log |a|)^{n_j}$$ with $$\chi_0(x)=\frac {x}{|x|}\hspace{0.5cm}\textrm{if }F=\mathbb{R}$$ and $$\chi_0(x)=\frac {x}{|x|^{1/2}}\hspace{0.5cm}\textrm{if }F=\mathbb{C},$$ $m_j$ is an integer, $n_j\geq 0$ an integer, and $\lambda_j$ is a complex number.

After a change of variables, this integral becomes \begin{multline}\int_{K_{n-1}}\int_{(F^\times)^{n-1}}\phi\left(a_1,a_2,\ldots,a_{n-1},\kappa(c,k)\right)\mu_1\left(a_1\right)\mu_2\left(a_2\right)\cdots \mu_{n-1}(a_{n-1})d^\times a_1d^\times a_2\cdots d^\times a_{n-1}\,dc.\nonumber\end{multline}

Looking just at the inner integral
\begin{multline}\int_{(F^\times)^{n-1}}\phi\left(a_1,a_2,\ldots,a_{n-1},\kappa(c,k)\right)\mu_1\left(a_1\right)\mu_2\left(a_2\right)\cdots \mu_{n-1}(a_{n-1})d^\times a_1d^\times a_2\cdots d^\times a_{n-1}, \label{precintegral}\end{multline}
 we have a Mellin Transform-style integral.

By Proposition \ref{expansion}, this converges absolutely for $\Re(s)$ sufficiently large and has a meromorphic continuation to the entire complex plane. 
In particular, for any positive real number $M$, there exists polynomial $Q_{\xi,M}(s)$, depending only on $M$ and the $\mu_j$ (which are determined by the $\xi$ occurring in Lemma \ref{gauge}),  and for each $\phi_\xi$, $c\in K_{n-1}$ and $k\in K_n$,  there exists a function $H_{\phi_\xi,\kappa(c,k)}(s)$ which is holomorphic in $s$ for $\Re(s)>-M$ such that (\ref{precintegral}) equals $$\frac{H_{\phi_\xi,\kappa(c,k)}(s)}{Q_{\xi,M}(s)}\hspace{0.5cm}\textrm{for }\Re(s)>-M.$$

By the continuity result (\ref{continuity}) of Proposition \ref{expansion}, we see that $H_{\phi_\xi,\kappa(c,k)} (s)$ varies continuously with respect to $\kappa(c,k)$  as a function of $s$ in the topology of uniform convergence on compact sets of the right half plane $\Re(s)>-M$.

Let $Q_M(s)$ be the least common multiple of the polynomials $Q_{\xi,M}(s)$ (recall that there were only finitely many such $\xi$ occurring in the expansion of $\Lambda_{n}(W,\cdot)$ and they are independent of the choice of $W$), and setting $$H_{W,\kappa(c,k)}(s)=\sum_\xi \frac{Q_M(s)}{Q_{\xi,M}(s)}H_{\phi_\xi,\kappa(c,k)}(s).$$ 

Then we have \begin{equation}J_1(s,k,\chi,W)=\frac{\int_{K_{n-1}} H_{W,\kappa(c,k)} (s) \,dc}{Q_M(s)} \label{integralc}\end{equation}
where $H_{W,\kappa(c,k)}$ is meromorphic in the right half plane $\Re(s)>-M$ and $H_{W,\kappa(c,k)}$ is holomorphic in $s$ for $\Re(s)>-M$ and depends continuously on $\kappa(c,k)$ in the topology of uniform convergence on compact sets in the half plane $\Re(s)>-M$.

Then $$H_1(s,k,W)=\int_{K_{n-1}} H_{W,\kappa(c,k)}(s)\, dc$$ defines a function which is again holomorphic in $s$ for $\Re(s)>-M$ and also depends continuously on $k$ in the topology of uniform convergence in compact sets in the half plane $\Re(s)>-M$.

Thus, after integrating with respect to $c$ we obtain $$J_1(s,k,\chi,W)= \frac{H_1(s,k,W)}{Q_M(s)}$$ which gives  meromorphic continuation of $J_1(s,k,\chi,W)$ to the right half plane $\Re(s)>-M$. Since $M$ is arbitrary, this gives a meromorphic continuation of $J_1(s,k,\chi,W)$ to the whole complex plane.

Now suppose $h\in GL_n$. Using a variation of the Iwasawa decomposition of $GL_n$, we can write $h=nazk$ where $n\in N_{n-1}$, $a\in A_{n-1}$, $z\in Z_n$, the center of $GL_n$, and $k\in K_n$. Assume $z=\textrm{diag}(b,\ldots, b)$ with $b\in F^\times$. If we further choose $b$ to be positive real, which is always possible, the choice of $b$ is uniquely and continuously determined by $h$. In fact, $$b=||\epsilon_n h||$$ where $||\cdot ||$ is the usual euclidean norm on $F^n$.

Let $\omega$ denote the central character of the representation $\pi$. Then, from the invariance of $J(s,h,\chi,W)$ under left translation of $h$ by elements of the mirabolic subgroup $P_n$, and the transformation property of $W$ under the center of $GL_{r}$, we have $$J_1(s,h,\chi,W)=\omega\chi^n(b)|b|^{n(s-1)}J_1(s,k,\chi,W).$$

In this manner we can define the meromorphic continuation of $J_1(s,h,\chi,W)$ for any $h$ in $GL_n$.

In particular, for $h\in GL_n$, $h=nazk$, we set 
$$H_1(s,h,W)=\omega\chi^n(\epsilon_n h)||\epsilon_n h||^{n(s-1)}H_1(s,k,W).$$ There is no ambiguity here since $||\epsilon_n k||=1$ for $k\in K_n$.
Then, given $M$ a positive real number, we have a polynomial $Q_M(s)$, independent of $W$ and $h$, and a function $H_1(s,h,W)$, holomorphic in $s$ for $\Re(s)>-M$ and depending continuously on $h$ in the topology of uniform convergence on compact sets in the right half plane $\Re(s)>-M$, such that $$J_1(s,h,\chi,W)= \frac{H_1(s,h,W)}{Q_M(s)}\hspace{0.5cm}\textrm{for }\Re(s)>-M.$$

We have proved the following:

\begin{prop}\label{meroJ1} For any $W\in\mathcal{W}(\pi,\psi)$ and $h\in GL_n$, the integral $J_1(s,h,\chi,W)$, which converges absolutely for $\Re(s)\gg 0$, has a meromorphic continuation to the whole complex plane.

In particular, given $M$, a positive real number, there exists a polynomial $Q_M$, independent of $W$ and $h$,  and a complex valued function $H_1(s,h,W)$ defined and holomorphic in $s$ on the right half plane $\Re(s)>-M$, such that
 \begin{equation}J_1(s,h,\chi,W)= \frac{H_1(s,h,W)}{Q_M(s)}\label{J1mero}\hspace{0.5cm}\textrm{for }\Re(s)>-M,\end{equation} where $H_1(s,h,W)$  varies continuously with $h$, as a function of $s$ in the topology of uniform convergence on compact sets in the right half plane $\Re(s)>-M$.
\end{prop}

\subsection{Proof of Theorem \ref{main} in the Archimedean Case}

Now we are ready to prove Theorem \ref{main}.

\begin{proof}[Proof of Theorem \ref{main}]

Recall that for $\Re(s)\gg0$ we have
\begin{equation}J(s,\chi,\Phi,W)=\int_{P_n\backslash GL_n} J_1(s,h,\chi,W)\Phi(\epsilon_n h)dh.\label{blah}\end{equation}

Fix $M$ a positive real number. Using (\ref{J1mero}) given in Proposition \ref{meroJ1}, we can write the right hand side of (\ref{blah})
as: \begin{equation}\frac {1}{Q_M(s)}\int_{P_n\backslash GL_n}H_1(s,h,W)\Phi(\epsilon_n h) dh\hspace{0.5cm}\textrm{for }\Re(s)>-M.\label{startJ}\end{equation}
Again, we emphasize that the function $H_1(s,h,W)$  varies continuously with $h$, as a function of $s$, in the topology of uniform convergence on compact sets in the right half plane $\Re(s)>-M$. Furthermore, the map $[h]\mapsto \epsilon_n h  \,:P_n\backslash Gl_n\rightarrow F^n\backslash \{0\}$ is a homeomorphism. Thus, if $\Phi$ has compact support in $F^n\backslash \{0\}$, then the function $h\mapsto \Phi(\epsilon_n h)$ has support in $GL_n$ which is compact modulo $P_n$. Therefore, the integral in (\ref{startJ}) converges absolutely for $\Phi$ with compact support in $F^n\backslash \{0\}$, and so defines a function $H(s,W,\Phi)$ which is holomorphic in $s$ for $\Re(s)>-M$. 

Since $M$ was arbitrary, this gives the meromorphic continuation of $J(s,\chi,\Phi,W)$ to the whole complex plane. 

Now, fix a complex number $s_0$ and choose $M$ with $-M<\Re(s_0)$. 

Take $W$ to be the Whittaker function from Lemma \ref{J1}. Then $J_1(s,1_n,\chi,W)$ is holomorphic as a function of $s$ and non-zero at $s=s_0$. That is, $$J_1(s,1_n,\chi,W)=\frac{H_1(s,1_n,W)}{Q_M(s)}$$ is holomorphic for all $s$ and nonzero for $s=s_0$.

If $s_0$ is not a root of $Q_M(s)$, this implies that $H_1(s_0,1_n,W)\neq 0$. By choosing $\Phi$ to be nonnegative real valued with compact support near $\epsilon_n$, so that the support of $\Phi(\epsilon_n \cdot)$ has compact support modulo $P_n$ near the identity in $GL_n$, we have $$H(s_0,W,\Phi)=\int_{P_n\backslash GL_n} H_1(s_0,h,W) \Phi(\epsilon_n h)dh\neq 0$$ and hence $J(s_0,\chi,\Phi,W)\neq 0.$

On the other hand, if $s_0$ is a root of $Q_M(s)$, then, since $J_1(s_0,1_n,\chi,W)$ is nonzero and finite, (\ref{J1mero}) implies that $H_1(s,1_n,W)$ has a zero at $s=s_0$ of the same order. Let $b$ be the order of this zero. Then we can write $$H_1(s,1_n,W)=(s-s_0)^b G_1(s,1_n,W)$$ where $G_1(s,1_n,W)$ is holomorphic in $s$ for $\Re(s)>-M$ and $G_1(s_0,1_n,W)\neq 0$.

By the Cauchy Integral Formula, if we take $\Delta$ to be a closed disk in the half plane $\Re(s)>-M$ centered at $s_0$, with radius $\varrho$ small enough so that $Q_M(s)$ and $H_1(s,1_n,W)$ have no other zeros in $\Delta$, and we let $\gamma$ denote the boundary of $\Delta$, then $$G_1(s_0,1_n,W)=\frac {1}{2\pi i}\oint_\gamma \frac {H_1(s,1_n,W)}{(s-s_0)^{b+1}}\, ds.$$

Choose a positive number $\eta$ with $$\eta<\varrho^{b}|G_1(s_0,1_n,W)|.$$

Since the continuity of $H_1(s,h,W)$ is uniform on compact sets such as $\gamma$, there is an open neighborhood $V$ of $1_n$ in $GL_n(F)$ such that $$|H_1(s,h,W)-H_1(s,1_n,W)|<\eta$$ for $h\in V$ and $s\in \gamma$. In fact, since $H_1(s,h,W)$ is invariant under left translation of $h$ by elements of $P_n$, the above holds for all $s\in \gamma$ and $h\in U=P_n V$.

Now, we choose $\Phi$ to be nonegative real valued with compact support, such that the support of $\Phi(\epsilon \cdot)$ is contained in $U$, and such that $$\int_{P_n\backslash GL_n} \Phi(\epsilon_n h)\, dh=1.$$

Then, for this type of $\Phi$ and $s\in \gamma$, 
\begin{IEEEeqnarray}{rCl}\IEEEeqnarraymulticol{3}{l}{|H(s,W,\Phi)-H_1(s,1_n,W)|}\nonumber \\
\quad&=&\left|\int_{P_n\backslash GL_n} H_1(s,h,W)\Phi(\epsilon_n h)\, dh-H_1(s,1_n,W)\right|\nonumber \\
&=&\left|\int_{P_n\backslash GL_n} H_1(s,h,W)\Phi(\epsilon_n h)\, dh-\int_{P_n\backslash GL_n}H_1(s,1_n,W)\Phi(\epsilon_n h)\, dh\right|\nonumber \\
&\leq& \int_{P_n\backslash GL_n}\left|H_1(s,h,W)-H_1(s,1_n,W)\right|\Phi(\epsilon_n h)\, dh\nonumber \\
&<&\eta \int_{P_n\backslash GL_n} \Phi(\epsilon_n h) dh\nonumber \\
&=&\eta. \end{IEEEeqnarray}

Therefore, 
\begin{IEEEeqnarray}{rCl}\IEEEeqnarraymulticol{3}{l}{\left|\frac{1}{2\pi i}\oint_\gamma \frac{H(s,W,\Phi)}{(s-s_0)^{b+1}}\, ds-G_1(s_0,1_n,W)\right|}\nonumber \\ \quad
&=&\left|\frac{1}{2\pi i}\oint_\gamma \frac{H(s,W,\Phi)}{(s-s_0)^{b+1}}\, ds-\frac {1}{2\pi i}\oint_\gamma\frac{H_1(s,1_n,W)}{(s-s_0)^{b+1}}\, ds\right|\nonumber \\
&<&\frac {\eta}{2\pi} \oint_\gamma \frac {1}{|s-s_0|^{b+1}}\, d|s|\nonumber \\
&=&\frac {\eta}{\varrho^{b}}\nonumber \\
&<&|G_1(s_0,1_n,W)|. \nonumber \end{IEEEeqnarray}
Since $G_1(s_0,1_n,W)\neq 0$, this shows that $$\frac{1}{2\pi}\oint_\gamma \frac{H(s,W,\Phi)}{(s-s_0)^{b+1}}\, ds\neq 0$$ for this choice of $W$ and $\Phi$. In particular, $H(s,W,\Phi)$ has a zero of order at most $b$ at $s=s_0$.

Thus $$J(s_0,\chi,\Phi,W)=\frac {H(s_0,W,\Phi)}{Q_M(s_0)}\neq 0.$$

This completes the proof.

\end{proof}

Remark: If $s=s_0$ is a root of $Q_M(s)$ of order $b$, we have not excluded the possibility that, in the above proof, the order of the root in $H(s,W,\Phi)$ is strictly less than $b$. If this is the case, $J(s,\chi,\Phi,W)$ will have a pole at $s=s_0$.

\section{The Nonarchimedean Case}

We assume throughout this section that $F$ is nonarchimedean. Let $\mathcal{O}$ denote the ring of integers in $F$, $\wp$ the unique maximal ideal in $\mathcal{O}$, and let $\varpi$ be the uniformizer of $\wp$. Let $q={|\varpi|}^{-1}$.
In this section, we will prove Theorem \ref{main} in the nonarchimedean case.

\subsection{Proof of Lemma \ref{J1} in the Nonarchimedean Case}

Again, we find it convenient to replace $W$ in (\ref{j1}) with its right translate by $\sigma^{-1}$. Recall, for $\Re(s)\gg 0$, we have (\ref{jj1real}) 
\begin{equation}J(s,\chi,\Phi,\rho(\sigma^{-1})W)=\int_{P_n\backslash GL_n} J_1(s,h,\chi,\rho(\sigma^{-1})W)\Phi(\epsilon_n h)dh.\end{equation}
where
\begin{IEEEeqnarray}{rCl}\IEEEeqnarraymulticol{3}{l}{J_1(s,h,\chi,\rho(\sigma^{-1})W)=
\chi(\det h)|\det h|^{s-1}}\nonumber\\ \quad & &\times \int_{N_{n-1}\backslash GL_{n-1}}\int_{\mathfrak{p}_{0,n}\backslash M_n}W\left[\sigma\left(\begin{array}{cc}1_n&Z\\0&1_n\end{array}\right)\left(\begin{array}{cccc} g &0&0&0\\0&1&0&0\\0&0&g&0\\0&0&0&1\end{array}\right)\left(\begin{array}{cc} h &0\\0&h\end{array}\right)\sigma^{-1}\right]\nonumber \\ & & \times\psi(-\textrm{Tr}(Z))dZ\,\chi(\det g)|\det g|^{s-1}dg.\label{j1na}\end{IEEEeqnarray}

By appropriately modifying the work in the previous section, we can show that the integrals $J(s,\chi,\Phi,W)$ and $J_1(s,h,\chi,W)$ both have meromorphic continuations in terms of rational functions of $q^{-s}$. Then we could modify the descending induction arguments used to prove Proposition 3, Section 7 of \cite{jasha} to obtain a result similar to Theorem \ref{main} in the nonarchimedean case involving meromorphic continuations. However, we wish to prove Lemma \ref{J1}, which gives $J_1(s,1_n,\chi,W)$ as an entire function for at a least one choice of $W$, using techniques similar to those used in the archimedean case. We will do this using Lemma \ref{smoothvector} and an explicit choice of function $\phi_0$.

\begin{proof}[Proof of Lemma \ref{J1} in the Nonarchimedean Case.]
We will prove this lemma by constructing an appropriate Whittaker function $W_0$.

Choose an integer $m>0$ sufficiently large so that the additive character $\psi$ is trivial on $\wp^m$, and that the multiplicative character $\chi$ is trivial on $1+\wp^m$. We denote by $K_j$ the compact open subgroup $1_j+M_j(\wp^m)$ of $GL_j(F)$.
In particular, for all $g\in K_{n-1}$, $\chi(\det g)=1$ and $|\det g|=1$. Furthermore, if $Z\in M_{n-1}(\wp^m)$, then $\psi(-\textrm{Tr}(Z))=1$.

Now let $\phi_0$ denote the smooth function on $GL_{r-1}(F)$ defined by 
$$\phi_0(g)=c\int_{N_{r-1}} \textrm{ch}_{K_{r-1}}(ug)\overline{\theta_{r-1}(u)}du$$
where $\textrm{ch}_{K_{r-1}}$ is the characteristic function of $K_{r-1}$ and $c=\frac{1}{\textrm{vol} (N_{r-1}\cap K_{r-1})}.$
Then $\phi_0$ is in the space $C_c^\infty(\theta_{r-1},GL_{r-1})$. In particular, $\phi_0$ is invariant by translation on the right by elements of $K_{r-1}$. In fact, 
%
\begin{equation}\phi_0(g)=\left\{\begin{array}{ll}\theta_{r-1}(u) & \textrm{if } g=uk\textrm{ for }u\in N_{r-1}\textrm{, } k\in K_{r-1}\\ 0 & \textrm{otherwise}.\end{array}\right.\label{phinot}\end{equation}

Now, by Proposition \ref{smoothvector}, there exists a $W_0\in\mathcal{W}(\pi,\psi)$ such that $W_0\left(\begin{array}{cc} g&0\\0&1\end{array}\right)=\phi_0(g)$ 
for all $g\in GL_{r-1}$.

Now, once again, we decompose $M_n(F)$ as $M_{n-1}(F)\oplus F^{n-1} \oplus F^{n-1}\oplus F$, writing $Z=\left(\begin{array}{cc} Z'&u\\Y& x\end{array}\right)$
with $Z'\in M_{n-1}$, $Y\in F^{n-1}$ as a row vector, $u\in F^{n-1}$ as a column vector and $x\in F$. Under this decomposition, the subspace $\mathfrak{p}_{0,n}$ decomposes as $\mathfrak{p}_{0,n-1}\oplus \{0\}\oplus F^{n-1} \oplus F$. Thus, the quotient space $\mathfrak{p}_{0,n}\backslash M_n$ decomposes as $\mathfrak{p}_{0,n-1}\backslash M_{n-1}\oplus F^{n-1}\oplus \{0\} \oplus \{0\}$. So, we write $Z=\left(\begin{array}{cc} Z'&0\\Y&0\end{array}\right)$ and as $Z$ ranges over $\mathfrak{p}_{0,n}\backslash M_n$, $Z'$ ranges over $\mathfrak{p}_{0,n-1}\backslash M_{n-1}$ and $Y$ ranges over $F^{n-1}$.

Also write $\sigma=\left(\begin{array}{cc}\sigma'&0\\0&1\end{array}\right)$ with $\sigma'\in GL_{r-1}$. Notice that $$\sigma\left(\begin{array}{cc}1_n&Z\\0&1_n\end{array}\right)\left(\begin{array}{cccc} g &0&0&0\\0&1&0&0\\0&0&g&0\\0&0&0&1\end{array}\right)\sigma^{-1}
=\left(\begin{array}{cc}\sigma'\left(\begin{array}{ccc}g&0&Z'g\\0&1&Yg\\0&0&g\end{array}\right){\sigma'}^{-1} &0\\0&1\end{array}\right)$$
which is in $GL_{r-1}$ (as a subgroup of $GL_{r}$).

Thus, taking $W=\rho(\sigma)W_0$ and $h=1_n$ in (\ref{j1na}),
and after a change of variable ($Yg \mapsto Y$), we have 
\begin{multline}J_1(s,1_n,\chi,W_0)= \\ \int_{N_{n-1}\backslash G_{n-1}}\int_{F^{n-1}}\int_{\mathfrak{p}_{0,n-1}\backslash M_{n-1}}\phi_0\left[\sigma'\left(\begin{array}{ccc}g&0&Z'g\\0&1&Y\\0&0&g\end{array}\right){\sigma'}^{-1}\right]\psi(-\textrm{Tr}(Z'))dZ'\,dY \chi(\det g)|\det g|^{s-2}dg.\label{ourcasena1}\end{multline}

For convenience, let $\gamma(Z',Y,g)$ denote the matrix $\sigma'\left(\begin{array}{ccc}g&0&Z'g\\0&1&Y\\0&0&g\end{array}\right){\sigma'}^{-1}.$

Now, assume $\phi_0\left[\gamma(Z',Y,g)\right]\neq 0.$
Then by (\ref{phinot}), $\gamma(Z',Y,g)=uk$ for some $u\in N_{r-1}$ and $k\in K_{r-1}$.
That is, setting $u'=\sigma'^{-1}u^{-1}\sigma'$,
 $$\sigma'u'\left(\begin{array}{ccc}g&0&Z'g\\0&1&Y\\0&0&g\end{array}\right){\sigma'}^{-1}\in K_{r-1}$$
or $$u'\left(\begin{array}{ccc}g&0&Z'g\\0&1&Y\\0&0&g\end{array}\right)\in {\sigma'}^{-1}K_{r-1}\sigma'=K_{r-1}.$$
 Note that the matrix $u'$ has the form $$u'=\left(\begin{array}{ccc} n_1& u_1& T\\0&1&0 \\ U& u_2 & n_2\end{array}\right)$$ where $n_i\in N_{n-1}$, $u_i\in F^{n-1}$ as column vectors, $T\in\mathfrak{p}_{0,n-1}$ and $U\in \mathfrak{u}_{0,n-1}$, the space of upper triangular matrices with $0$'s along the diagonal.

Using this expression for $u'$, $u'\left(\begin{array}{ccc}g&0&Z'g\\0&1&Y\\0&0&g\end{array}\right)$ takes the form 
\begin{equation}\left(\begin{array}{ccc}n_1 g& u_1 & n_1Z'g+u_1 Y+Tg\\0&1&Y\\Ug& u_2& UZ'g+u_2Y+n_2g\end{array}\right)\label{ink}\end{equation}
which, if $\phi_0[\gamma(Z',Y,g)]\neq 0$, is in $K_{r-1}$.
We see from the top left corner of this matrix that $g$ must be of the form $u_0 k_0$, with $u_0\in N_{n-1}$ and $k_0\in K_{n-1}$ and from the last entry in the center row that $Y\in \oplus_{i=1}^{n-1}\wp^m$. Thus the outer most integral in (\ref{ourcasena1}) can be taken over the image of $K_{n-1}$ in $N_{n-1}\backslash GL_{n-1}$ (this image will be compact) and the integral in $Y$ can be assumed to be over $\oplus_{i=1}^{n-1} \wp^m$, a compact set.

Also, in this case, $\chi(g)=1$ and $| \det g|=1$ for all $g$ for which $\phi_0[\gamma(Z',Y,g)]\neq 0$. 

 Note that, in this case,
\begin{IEEEeqnarray}{rCl}\gamma(Z',Y,u_0k_0)
&=&\sigma'\left(\begin{array}{ccc}u_0&0&Z'u_0\\0&1&0\\0&0&u_0\end{array}\right){\sigma'}^{-1}\sigma'\left(\begin{array}{ccc}k_0&0&0\\0&1&Y\\0&0&k_0\end{array}\right){\sigma'}^{-1}.\nonumber \end{IEEEeqnarray} Here, the matrix $\sigma'\left(\begin{array}{ccc}k_0&0&0\\0&1&Y\\0&0&k_0\end{array}\right){\sigma'}^{-1}$ is in $K_{r-1}$, and $\phi_0$ is invariant on the right by this subgroup. Thus, if $\phi_0[\gamma(Z',Y,g)]\neq 0$ then $\phi_0[\gamma(Z',Y,g)]=\phi_0[\gamma(Z',0,u_0)].$

Furthermore, notice that the expression
$\phi_0[\gamma(Z',0,g)]\psi(-\textrm{Tr}(Z'))$ is invariant under translating $g$ on the left by elements of $N_{n-1}$. (The matrix $\sigma' \left(\begin{array}{ccc} u_0 &0&0\\0&1&0\\0&0&u_0\end{array}\right)\sigma'^{-1}$ is in $N_{r-1}$ and is mapped to $1$ by $\theta_{r-1}$. Also, $\textrm{Tr}$ is invariant under conjugation.) Therefore, (\ref{ourcasena1}) is the same as
$$\int_{(N_{n-1}\cap K_{n-1})\backslash K_{n-1}}\int_{\oplus_{i=1}^n \wp^m}\int_{\mathfrak{p}_{0,n-1}\backslash M_{n-1}}\phi_0\left[\gamma(Z',0,1_{n-1})\right]\psi(-\textrm{Tr}(Z'))dZ'\,dY\,dg.$$
The integrals over $g$ and $Y$ come out as positive constants and so, 
\begin{equation}J_1(s,1_n,\chi,W_0)=c\int_{\mathfrak{p}_{0,n-1}\backslash M_{n-1}}\phi_0\left[\gamma(Z',0,1_{n-1})\right]\psi(-\textrm{Tr}(Z'))dZ',\label{J1justZ'}\end{equation}
 for some $c\neq 0$.

Now, we can evaluate this integral explicitly. Again, assuming that $\phi_0[\gamma(Z',0,1_{n-1})]$ is not zero, $\gamma(Z',0,1_{n-1})\in N_{r-1} K_{r-1}.$
In this case, the matrix (cf. (\ref{ink}))
$\left(\begin{array}{ccc}n_1& u_1 & n_1Z'+T\\0&1&0\\U& u_2& UZ'+n_2\end{array}\right)$ is in $K_{r-1}$, where $n_i\in N_{n-1}$, $u_i\in F^{n-1}$ as column vectors, $T\in\mathfrak{p}_{0,n-1}$ and $U\in \mathfrak{u}_{0,n-1}$.
Looking at the upper right corner, we have $n_1Z'+T\in M_{n-1}(\wp^m)$. But $n_1$, in the upper left corner, is in fact in $N_{n-1}\cap K_{n-1}$, so that $Z'+n_1^{-1}T$ is in $M_{n-1}(\wp^m)$. Since $n_1^{-1}T$ is upper triangular, the entries below the diagonal in $Z'$ must be in $\wp^m$. Thus $Z'$, as an element of $\mathfrak{p}_{0,n-1}(F)\backslash M_{n-1}(F)$, is in the image of $M_{n-1}(\wp^m)$ in that quotient space. Therefore, the integral (\ref{J1justZ'}) can be taken over the compact space $\mathfrak{p}_{0,n-1}(\wp^m)\backslash M_{n-1}(\wp^m)$.

In this case, $\psi(-\textrm{Tr}(Z'))=1$ from our choice of $m$ and $\gamma(Z',0,1_{n-1})\in K_{r-1}$. Thus $\phi_0[\gamma(Z',0,1_{n-1})]\psi(-\textrm{Tr}(Z'))=1.$ Thus, (\ref{J1justZ'}) becomes $$J_1(s,1_n,\chi,W_0)=c\int_{\mathfrak{p}_{0,n-1}(\wp^m)\backslash M_{n-1}(\wp^m)} dZ'\neq 0.$$ Thus, as a function of $s$, $J_1(s,1_n,\chi,W_0)$ is in fact a non-zero constant. In particular, it is holomorphic in $s$, and $J_1(s_0,1_n,\chi,W_0)\neq 0$.

\end{proof}

\subsection{Proof of Theorem \ref{main} in the Nonarchimedean Case}

Now we can prove Theorem \ref{main} in the nonarchimedean case. The proof uses the same idea as that in the archimedean case, but is somewhat simpler. 

\begin{proof}[Proof of Theorem \ref{main} in the Nonarchimedean Case.] Let $W_0$ be the Whittaker function given in Lemma \ref{J1} for which $J_1(s,1_n,\chi,W_0)$ is holomorphic and $J_1(s_0,1_n,\chi,W_0)\neq 0$ (in fact, from the proof of Lemma \ref{J1} in the nonarchimedean case, we may assume it is a non-zero constant with respect to $s$).  

Using the smoothness of $W_0$, we can find an integer $m>0$ such that $W_0$ is right invariant under the compact subgroup $$\mathcal{K}_{n,m}=\left\{\left.\left(\begin{array}{cc}k&0\\0&k\end{array}\right)\right| k\in K_{n,m} \right\}$$
where
$$K_{n,m}=1_{n}+M_{n}(\wp^m)$$ 
and $\wp$ is the unique maximal ideal in $F$.

Since $|\det k|=1$ for all $k\in K_{n,m}$, this gives us, for $\Re(s)\gg0$,  $$J_1(s,k,\chi,W_0)=J_1(s,1_n,\chi,W_0)$$ for all $k\in K$.

Notice that the set $\epsilon_n K_{n,m}$ is compact and open in $F^n$ and does not contain the origin.

Choose $\Phi_{K_{n,m}}$ to be the characteristic function of $\epsilon_n K_{n,m}$. 
Recall that, for $h\in GL_n$, the coset of $h$ in  $P_n\backslash GL_n$ is uniquely determined by $\epsilon_n h$. So, if $\epsilon_n h\in \epsilon_n K_{n,m}$, $J_1(s,h,\chi,W_0)=J_1(s,k,\chi,W_0)$ where $k\in K_{n,m}$ is such that $\epsilon_n k=\epsilon_n h$ (as a function of $h$, $J_1(s,h,\chi,W_0)$ is invariant under left translation by elements of $P_n$).
Therefore, 
$$J_1(s,h,\chi,W_0)\Phi_{K_{n,m}}(\epsilon_n h)=\left\{\begin{array}{ll} J_1(s,1_n,\chi,W_0)&\textrm{if }\epsilon_nh\in \epsilon_nK_{n,m},\\0&\textrm{otherwise}.\end{array}\right.$$

Thus, with $\Phi=\Phi_{K_{n,m}}$, (\ref{jj1real}) becomes
\begin{equation}J(s,\chi,\Phi_{K_{n,m}},W_0)=J_1(s,1_n,\chi,W_0)\textrm{vol}(\epsilon_n K_{n,m}).\label{j=j1}\end{equation}

This holds for $\Re(s)$ sufficiently large. However, the right hand side is holomorphic in $s$ by Lemma \ref{J1}, and therefore, (\ref{j=j1}) defines a $J(s,\chi,\Phi_{K_{n,m}}, W_0)$ as a holomorphic function of $s$

Since $\textrm{vol}(\epsilon_n K_{n,m})$ is non-zero, we see that $$J_1(s_0,1_n,\chi,W_0)\neq 0$$ implies that $$J(s_0,\chi,\Phi_{K_{n,m}},W_0)\neq 0.$$

This completes the proof of the theorem.

\end{proof}

\section{The Global Theory}

\subsection{Setup}\label{setup}

We wish to consider the global integral $I$ defined in sections 5 and 6 of \cite{jasha}. For the convenience of the reader, we repeat the full definition and recall some of the main results of those sections here.

Let $F$ be a number field with adele ring $\mathbb{A}_F$.

Let $\psi$ be a non-trivial additive character of $\mathbb{A}_F/ F$. We will consider an automorphic unitary cuspidal representation $\pi$ of $GL_r(\mathbb{A}_F)$. We denote by $\omega_\pi$ the central character associated to $\pi$.
Let $\chi$ be a unitary idele-class character of $F$.
Let $\phi$ be a form in the space of $\pi$.

Again we assume that $r=2n$ is an even integer. Let $M_n$ denote the ring of $n\times n$ matrices over $\mathbb{A}_F$. Let $P_{n-1,n}$ be the parabolic subgroup of type $(n-1,1)$ in $GL_n$, $A_n$ the group of diagonal matrices, $B_n$ the group of upper triangular matrices, $N_n$ the group of upper triangular matrices with unit diagonal, and $Z_n$ the center of $GL_n$.

Define $V_0$ to be the group of matrices of the from $$\left(\begin{array}{cc} 1_n& X\\0&1_n\end{array}\right)$$ where $X\in M_n$.

Define $$\theta'(v)=\psi(\textrm{Tr}X).$$
This defines a character on $V_0(\mathbb{A}_F)$ which is trivial on $V_0(F)$, and is fixed by conjugation by elements of the form: $$\left(\begin{array}{cc} g&0\\0&g\end{array}\right),\hspace{1cm}g\in GL_n(\mathbb{A}_F).$$

Let $\Phi$ be a Schwartz-Bruhat function in $n$ variables. Set
 \begin{equation}f(g,s)=\int_{\mathbb{A}_F} \Phi(\epsilon_n t g) {|t|}^{ns}\chi^n\omega_\pi (t) d^\times t\chi(\det g)|\det g|^s, \label{f}\end{equation}
 where $$\epsilon_n=(\underbrace{0,0,\ldots,0}_{n-1},1).$$

Define the Eisenstein series \begin{equation}E(g,s)=\sum_{\gamma\in P_{n-1,n}(F)\backslash GL_n(F)} f(\gamma g,s).\label{eis}\end{equation}

Finally, we define the integral $I$ as 
\begin{equation}I(s,\chi,\phi,\Phi)=\int_{GL_n(F)\backslash GL_n(\mathbb{A}_F)/Z_n(\mathbb{A}_F)}\int_{V_0(F)\backslash V_0(\mathbb{A}_F)} \phi\left[v\left(\begin{array}{cc} g&0\\0&g\end{array}\right)\right] \theta'(v)dv E(g,s) dg.\label{defI}\end{equation}

Section 5.2 of \cite{jasha} states that this integral converges for all $s$, and so defines a function which is holomorphic except for at those $s$ which are singularities of the Eisenstein series.

Recall that by Lemma 4.2 of \cite{prod1}, or rather the proof of that lemma, the Eisenstein series is holomorphic  on the entire complex plane unless $\chi^n\omega_\pi$ is trivial on the ideles of absolute value 1. In the case that $\chi^n\omega_\pi(\alpha)=|\alpha|^{i\tau}$ ($\tau$ a real number), the Eisenstein series extends meromorphically to the entire complex plane, with at most simple poles at $s=1-i\frac \tau n$ and $s=-i\frac \tau n$ (see equations (6) and (7) of the proof of Lemma 4.2 in \cite{prod1}).

Note that replacing $s$ with $s+i\tau_0$ in (\ref{f}) for some real $\tau_0$ amounts to changing $\chi$ to  $\chi_0=\chi|\cdot|^{i\tau_0}$. 
That is $$f(g,s+i\tau_0)=f_0(g,s)$$ where 
$$f_0(g,s)=\int_{\mathbb{A}_F} \Phi(\epsilon_n t g) {|t|}^{ns}\chi_0^n\omega_\pi (t) d^\times t\chi_0(\det g)|\det g|^s.$$
Thus, there is no loss in generality in assuming $\chi$ is such that $\chi^n\omega_\pi$ is either nontrivial on the ideles of absolute value one or it is the trivial character 1 (i.e., $\tau=0$). We will assume, from this point on, that this is case.

Hence, the Eisenstein series is holomorphic  on the entire complex plane unless $\chi^n\omega_\pi$ is trivial on the ideles of absolute value 1, in which case it is meromorphic with possible poles only at $s=1$ and $s=0$.

Therefore, $I(s,\chi,\phi,\Phi)$ is holomorphic on all of $\mathbb{C}$ unless $\chi^n\omega_\pi=1$, in which case it has simple poles at $s=1$ and $s=0$.

We denote by $\mathcal{W}(\pi,\psi)$ the Whittaker model of $\pi$ and let $W\in\mathcal{W}(\pi,\psi)$.

Define the global integral $$J=J(s,\chi,W,\Phi)$$ defined by \begin{equation}J=\int W\left[\sigma\left(\begin{array}{cc}1_n&Z\\0&1_n\end{array}\right)\left(\begin{array}{cc} g &0\\0&g\end{array}\right)\right]\psi(-\textrm{Tr}(Z))dZ\,\Phi(\epsilon_n g)\chi(\det g)|\det g|^sdg.\label{Jglobaldef}\end{equation}
Here, $Z$ is integrated over the quotient $$\mathfrak{p}_{0,n}(\mathbb{A}_F)\backslash M_n(\mathbb{A}_F),$$ and $g$ is integrated over the quotient $$N_n(\mathbb{A}_F)\backslash GL_n(\mathbb{A}_F).$$
This integral converges absolutely for $\Re(s)\gg 0$.

Proposition 5 of section 6 of \cite{jasha} states:

\begin{prop} For $\Re(s)\gg0$, $$I(s,\chi,\phi,\Phi)=J(s,\chi,W,\Phi)$$
where $$W(g)=\int_{N_r(F)\backslash N_r(\mathbb{A}_F)} \phi(ug)\theta_r(u)du$$
and $\theta_r$ is defined by $$\theta_r(u)=\prod_{j=1}^{r-1} \psi(u_{j,j+1}).$$
\end{prop}

Now, assume that $\phi$ is a smooth vector in the space of $\pi$ such that the associated Whittaker function $W$ is a product of local Whittaker functions: $$W=\prod_\nu W_\nu.$$
Assume, also that $\Phi$ is a Schwartz-Bruhat function in $n$ variables which is the product of local functions: $$\Phi=\prod_\nu \Phi_\nu.$$
Furthermore, we can write $$\chi=\prod_\nu \chi_\nu$$
and $$\psi=\prod_\nu \psi_\nu.$$

Then, in this case $$J(s,\chi,W,\Phi)=\prod_\nu J(s,\chi_\nu,W_\nu,\Phi_\nu),$$ where the local integrals are defined as in (\ref{Jdef}).

In other words, we have \begin{equation}I(s,\chi,\phi,\Phi)=\prod_\nu J(s,\chi_\nu,W_\nu,\Phi_\nu)\label{euler}\end{equation} for $\Re(s)\gg0.$

\subsection{The Partial L-function}

We recall the definition of the standard local exterior square $L$-function. We keep the same notation as from the previous subsection.
 
Let $S$ be a finite set of places including all archimedean places and all of the ramified nonarchimedean places for $\pi$ and $\chi$. 
For a nonarchimedean place $\nu$,  unramified for both $\pi$ and $\chi$, let $\mathcal{O}_\nu$ denote the ring of integers of $F_\nu$. Let $\varpi_\nu$ denote the generator of the unique maximal ideal $\wp_\nu$ in $\mathcal{O}_\nu$, and denote by $q_\nu$ the cardinality of the residue field. Let $A_\nu\in GL_r(\mathbb{C})$ denote the so-called Langlands class of the representation $\pi_\nu$, a certain conjugacy class in $GL_r(\mathbb{C})$. Then
$$L_\nu(s,\pi_\nu, {\bigwedge}^2 \otimes \chi_\nu)=\det(1-\chi_\nu(\varpi_\nu)q_\nu^{-s}{\bigwedge}^2(A_\nu))^{-1}.$$

Set $$L^S(s,\pi,{\bigwedge}^2\otimes \chi)=\prod_{\nu\notin S}L_\nu(s,\pi_\nu,{\bigwedge}^2\otimes \chi_\nu).$$
This infinite product converges absolutely for $\Re(s)\gg 0$.

We wish to prove the following theorem.

\begin{thm}\label{global} For $S$ a finite set of places including all archimedean places  and all of the ramified nonarchimedean places for $\pi$ and $\chi$, the partial $L$-function $$L^S(s,\pi,{\bigwedge}^2\otimes\chi)$$ extends to a meromorphic function on the whole complex plane. It is entire if $\chi^n\omega_\pi$ is nontrivial on the ideles of absolute value one. Otherwise, if $\chi^n\omega_\pi=1$, then $L^S(s,\pi,{\bigwedge}^2\otimes\chi)$ is holomorphic except for (possible) simple poles at $s=1$ and $s=0$.
\end{thm}

Remark: We assumed that the poles of the Eisenstein series (\ref{eis}), if they exist, are actually at $s=0$ and $s=1$ (i.e., $\tau=0$).
If this assumption is not made, then the poles of $L^S(s,\pi,{\bigwedge}^2\otimes\chi)$, if they exist, will be at $s=1-i\frac \tau n$ and $s=-i\frac \tau n$ where $\chi^n\omega_\pi=|\cdot|^{i\tau}$, $\tau$ a real number.

Recall that \cite{jasha} (see Theorem 1, Section 8 of that paper) gives a criterion for when the pole at $s=1$ (and hence at $s=0$) exists in terms of a certain period integral.

\begin{proof}

Recall that if $\nu \notin S$, then $\pi_\nu$ is a spherical representation in the sense that there exists a non-zero $K_\nu$ fixed vector in the space of $\pi_\nu$, $K_\nu$ denoting the usual maximal compact subgroup of $GL_r(F_\nu)$.

Choose $\phi$ from the space of $\pi$ such that the associated Whittaker function $W\in \mathcal{W}(\pi,\psi)$ is of the form $$W=\prod_\nu W_v$$ where, for $\nu\notin S$, $W_\nu$ is invariant under right translation by $K_\nu$ (i.e. $W_\nu$ is the Whittaker function associated to a $K_\nu$ fixed vector in the space of $\pi_\nu$) normalized to take value 1 on $K_\nu$.

Choose $\Phi$ a Schwartz-Bruhat function of the form $$\Phi=\prod_\nu \Phi_\nu$$ such that for $\nu\notin S$, $\Phi_\nu$ is the characteristic function of the lattice of integers in $F_\nu^n$.

Proposition 2, Section 7 of \cite{jasha} states that for $\nu\notin S$, $$J(s,\chi_\nu,W_\nu,\Phi_\nu)=L_\nu(s,\pi_\nu,{\bigwedge}^2\otimes \chi_\nu).$$ 
Thus, (\ref{euler}) becomes \begin{equation}I(s,\chi,\phi,\Phi)=L^S(s,\pi,{\bigwedge}^2\otimes \chi)\prod_{\nu\in S} J(s,\chi_\nu,W_\nu,\Phi_\nu).\label{LT}\end{equation}
for $\Re(s)\gg 0$.

The left hand side of (\ref{LT}) is meromorphic for all choices of $\phi$ and $\Phi$.

Fix $s_0\in \mathbb{C}$. By Theorem \ref{main}, for the finitely many $\nu$ in $S$, we may choose $\Phi_\nu$ and $W_\nu$ so that the local integrals $J(s,\chi_\nu,W_\nu,\Phi_\nu)$ can be meromorphically continued to the whole complex plane and $J(s_0,\chi_\nu,W_\nu,\Phi_\nu)\neq 0$.

In particular, the finite product $\prod_{\nu\in S} J(s,\chi_\nu,W_\nu,\Phi_\nu)$, as a function of $s$, is meromorphic and is not identically 0. Thus, (\ref{LT}) extends $L^S(s,\pi,{\bigwedge}^2\otimes \chi)$ meromorphically to all $s$.

Furthermore, since $$\prod_{\nu\in S}J(s_0,\chi_\nu,W_\nu,\Phi_\nu)\neq 0$$
 for this particular choice of $W$ and $\Phi$, (\ref{LT}) also shows that $L^S(s,\pi,{\bigwedge}^2\otimes \chi)$ has a pole of order $m$ at $s=s_0$ if and only if the global integral $I(s,\chi,\phi,\Phi)$ has a pole of order $m$ at $s=s_0$.

In the previous section, we have seen that $I(s,\chi,\phi,\Phi)$ is entire if $\chi^n\omega_\pi$ is nontrivial on the ideles of absolute value one, and has (possible) simple poles at $s=1$ and $s=0$ if $\chi^n\omega_\pi=1$.

As this can be done for all $s_0\in\mathbb{C}$, this proves our result.

\end{proof}

Remark: We remind the reader that the above result does not state that if $\chi^n\omega_\pi=1$, then $L^{S}(s,\pi,{\bigwedge}^2\otimes \chi)$ must necessarily have a pole at $s=1$ and $s=0$. Theorem 1 of section 8 of \cite{jasha} gives a more explicit criterion for the existence of a pole at $s=1$ if $\chi^n\omega_\pi=1$. Namely, if the integral 
$$\int_{GL_n(F)\backslash GL_n(\mathbb{A}_F)/Z(\mathbb{A}_F)} \int_{M_n(F)\backslash M_n(\mathbb{A}_F)}
\phi\left[\left(\begin{array}{cc} 1_n&X\\0&1_n\end{array}\right)\left(\begin{array}{cc} g&0\\0&g\end{array}\right)\right]
\psi(\textrm{Tr} X)dX\, \chi(\det g)dg$$
is non-zero for some $K_r$-finite vector $\phi$ in the space of $\pi$, then $L^{S}(s,\pi,{\bigwedge}^2\otimes\chi)$ will have a pole at $s=1$.

\section{The Odd Case}

\subsection{Analogous Results for the Odd Case}

We note here that the techniques used above also apply to the case of $GL_r$ where $r$ is odd. In some sense this case is easier, since the integrals involved do not depend, at least initially, on a Schwartz-Bruhat function as one of the parameters.

We outline those results already established in section 9 of \cite{jasha}, and discuss the analogues of our main results in this case. Assume thoughout this section that $r=2n+1$ is odd.

We begin with the global situation. Let $F$ be a number field. Again, $\pi$ is a unitary cuspidal representation of $GL_r(\mathbb{A}_F)$, and $\chi$ is a unitary gr\"ossencharacter. For $\phi$ in the space of $\pi$, we define the global integral 
\begin{multline}I(s,\chi,\phi)=
\int_{GL_n(F)\backslash GL_n(\mathbb{A}_F)} \int_{M_n(F)\backslash M_n(\mathbb{A}_F)}\int_{F^n\backslash \mathbb{A}_F^n} \phi\left[\left(\begin{array}{ccc}1_n&Z&Y\\0&1_n&0\\0&0&1\end{array}\right)\left(\begin{array}{ccc}g&0&0\\0&g&0\\0&0&1\end{array}\right)\right]\\
\times \psi(\textrm{Tr}(Z))dZ\,dY\,\chi(\det g){|\det g |}^{s-1} dg.\nonumber \end{multline}

Proposition 1, Section 9 of \cite{jasha} states that this integral converges absolutely in the sense that the integral 
$$\int \left|\int \int \phi\left[\left(\begin{array}{ccc}1_n&Z&Y\\0&1_n&0\\0&0&1\end{array}\right)\left(\begin{array}{ccc}g&0&0\\0&g&0\\0&0&1\end{array}\right)\right]
 \psi(\textrm{Tr}(Z))dZ\,dY\right|{|\det g |}^{\Re(s)-1} dg$$
conerges for all $s$.
Thus $I(s,\chi,\phi)$ defines a function of $s$ which is holomorphic for all $s$.

Again, $\theta_r$ is a character of $N_r$ defined by the an additive character $\psi$, and we let $W\in\mathcal{W}(\pi,\psi)$ be the Whittaker function attached to $\phi$. 

Here, we let $\sigma$ be the permutation matrix which takes the sequence
$$(1,2,3,\ldots, n, n+1, n+2, n+3,\ldots, 2n, 1)$$ to the sequence $$(1,3,5,\ldots, 2n-1, 2,4,6,\ldots, 2n, 1).$$

Define the Eulerian integral
 \begin{multline}J(s,\chi,W)=
\int_{N_n(\mathbb{A}_F)\backslash GL_n(\mathbb{A}_F)} \int_{\mathfrak{p}_{0,n}(\mathbb{A}_F)\backslash M_n(\mathbb{A}_F)} W\left[\sigma\left(\begin{array}{ccc}1_n&Z&0\\0&1_n&0\\0&0&1\end{array}\right)\left(\begin{array}{ccc}g&0&0\\0&g&0\\0&0&1\end{array}\right)\right]\\
\times \psi(-\textrm{Tr}(Z))dZ\,\chi(\det g){|\det g |}^{s-1} dg\nonumber \end{multline} with $\mathfrak{p}_{0,n}$ the space of upper triangular matrices in $M_n$.

Proposition 2, Section 9 of \cite{jasha} states that $J(s,\chi,W)$ converges absolutely for $\Re(s)\gg 0$ and gives the equality $$I(s,\chi,\phi)=J(s,\chi,W)$$ for $\Re(s)\gg 0$.

Now assume that $W=\prod_\nu W_\nu$.

As before, $J(s,\chi,W)$ is a product of local integrals, 
$$J(s,\chi,W)=\prod_\nu J(s,\chi_\nu,W_\nu)$$
where \begin{multline}J(s,\chi_\nu,W_\nu)=
\int_{N_n(F_\nu)\backslash GL_n(F_{\nu})} \int_{\mathfrak{p}_{0,n}(F_\nu)\backslash M_n(F_\nu)} W_\nu\left[\sigma\left(\begin{array}{ccc}1_n&Z&0\\0&1_n&0\\0&0&1\end{array}\right)\left(\begin{array}{ccc}g&0&0\\0&g&0\\0&0&1\end{array}\right)\right]\\
\times \psi_\nu(-\textrm{Tr}(Z))dZ\,\chi_\nu(\det g){|\det g |}^{s-1} dg.\label{jodd}\end{multline}

Proposition 3, Section 9 of \cite{jasha} states that for each place $\nu$, there exists an $\eta>0$ for which the local integral $J(s,\chi_\nu,W_\nu)$ converges absolutely for $\Re(s)>1-\eta$.

Finally, by Proposition 4, Section 9 of \cite{jasha}, at the unramified nonarchimedean places $\nu$ for $\pi$ and $\chi$, if we take $W_\nu$ to be the spherical element, that is, $W_\nu$ to be invariant on the right by the maximal subgroup $GL_r(\mathcal{O}_\nu)$, we have $$J(s,\chi_\nu,W_\nu)=L(s,\pi_\nu, {\bigwedge}^2\rho\otimes \chi_\nu).$$

Let $S$ be a finite set of places containing all the archimedean ones and those which are ramified for $\pi$ or $\chi$. 
If $$L^S(s,\pi,{\bigwedge}^2\rho\otimes\chi)=\prod_{\nu\notin S}L(s,\pi_\nu,{\bigwedge}^2\rho\otimes\chi_\nu),$$
then we have \begin{equation}I(s,\chi\phi)=L^S(s,\pi,{\bigwedge}^2\rho\otimes\chi)\prod_{n\in S} J(s,\chi_\nu,W_\nu).\label{lsodd}\end{equation}

We have the analogue of Theorem \ref{main}.

\begin{thm}\label{mainodd}Let $F$ be any local field. There exists $W_\nu$ such that $J(s,\chi_\nu,W_\nu)$ defines a holomorphic function for all $s$, and $$J(s_0,\chi_\nu,W_\nu)\neq 0.$$
\end{thm}

\begin{proof}Note the similarity between the function $J(s,\chi_\nu,W_\nu)$ and the function $J_1(s,1_n,\chi_\nu,W_\nu)$ defined in (\ref{j1}). The matrix $\sigma\left(\begin{array}{ccc}1_n&Z&0\\0&1_n&0\\0&0&1\end{array}\right)\left(\begin{array}{ccc}g&0&0\\0&g&0\\0&0&1\end{array}\right)$ is in $GL_{r-1}$, as a subgroup of $GL_r$ embedded in the usual way. Thus we may immediately apply Proposition \ref{smoothvector}, which holds for $r$ odd, to the integrand in (\ref{jodd}). The proof then follows almost exactly as that of Lemma \ref{J1}.

\end{proof}

Thus, by Theorem \ref{mainodd}, we can proceed as we did in the even case. For an appropriate choice of $W=\prod_\nu W_\nu$, (\ref{lsodd}) defines $L^S(s,\pi,{\bigwedge}^2\rho\otimes\chi)$ as a meromorphic function of $s$, which, in fact, is entire. Indeed, for each $s_0\in \mathbb{C}$, $W$ can be chosen so that $\prod_{n\in S} J(s_0,\chi_\nu,W_\nu)\neq 0$. $I(s,\chi,W)$ is holomophic at $s=s_0$ (for all choices of $W$), thus $L^S(s,\pi,{\bigwedge}^2\rho\otimes\chi)$ cannot have a pole at $s=s_0$.

Therefore we have:

\begin{thm}\label{mainglobalodd} Let $S$ be a finite set of places containing all the archimedean places and all those which ramify for $\pi$ or $\chi$. Then the function $L^S(s,\pi,{\bigwedge}^2\rho\otimes\chi)$ extends holomorphically to all of $\mathbb{C}$.
\end{thm}

Compare this result to \cite[Theoreom 3.5]{kimcan}, regarding the full (or completed) $L$-function.

\bibliography{nonvanishing} 
\bibliographystyle{plain}

\newpage

\appendix

 \section{Taylor expansions and Wirtinger Derivatives}

Given a smooth function $f:\mathbb{C}\rightarrow \mathbb{C}$, thought of as a smooth function from $\mathbb{R}^2$ into $\mathbb{C}$, taking $z=x+iy$, we abuse the notation and write $f(x,y)=f(z)$. We have the Wirtinger derivatives, formally defined as $$\frac{\partial}{\partial z}=\frac 12\frac {\partial}{\partial x}-\frac i2\frac {\partial}{\partial y}$$ and $$\frac {\partial}{\partial\overline{z}}=\frac 12 \frac {\partial}{\partial x}+\frac i2 \frac {\partial}{\partial y}.$$

They satisfy all the expected properties of differential operators. In particular, they satisfy a product rule and a chain rule. See \cite [p. 4] {hormander} for a brief summary of these properties.

Of particular interest to our purposes, is the fact that they commute with one another, in the sense $$\frac {\partial^2}{\partial z\partial \overline{z}}=\frac {\partial^2}{\partial \overline{z}\partial z},$$ as can be readily verified by direct computation.

We wish to prove the following analog of the Taylor Expansion Theorem and the Integral Form of the Remainder.

\begin{prop}\label{ztaylor}
Given a smooth function $f:\mathbb{C}\rightarrow \mathbb{C}$, and a positive integer $k$, the following holds for all $z\in \mathbb{C}$.
\begin{IEEEeqnarray}{rCl}f(z)&=&\sum_{l=0}^k\sum_{j=0}^l \frac {1}{j!(l-j)!}z^j\overline{z}^{l-j} \frac {\partial^k f}{\partial z^j \partial \overline{z}^{k-j}}(0)\nonumber \\
&&+\sum_{j=0}^{k+1}\frac {k+1}{j!(k+1-j)!}z^j\overline{z}^{k+1-j}\int_0^1(1-t)^k \frac {\partial ^{k+1}f}{\partial z^j\partial \overline{z}^{k+1-j}}(tz)dt.\nonumber \end{IEEEeqnarray}
\end{prop}

First, we need two lemmas:

\begin{lem}
If $D_1$ and $D_2$ are two operators wich commute, the binomial theorem $$(D_1+D_2)^l=\sum_{j=0}^l \left(\begin{array}{c} l\\ j\end{array}\right)D_1^j D_2^{l-j}$$ holds.
\end{lem}

\begin{proof} This is straightforward.
\end{proof}

From this we obtain:

\begin{lem}\label{taylorequation} Given $f:\mathbb{C}\rightarrow \mathbb{C}$ and setting $z=x+iy$, then the following holds for all positive integers $l$ and all $\tau=a+ib\in \mathbb{C}$.
$$\sum_{j=0}^l \frac{1}{j!(l-j)!}a^j b^{l-j}\frac{\partial^l f}{\partial x^j\partial y^{l-j}}(z)=\sum_{j=0}^l \frac{1}{j!(l-j)!}\tau^j \overline{\tau}^{l-j}\frac{\partial^l f}{\partial x^j\partial \overline{z}^{l-j}}(z).$$
\end{lem}

\begin{proof}
By direct computation $$a\frac{\partial }{\partial x}+b\frac{\partial}{\partial y}=\tau\frac{\partial }{\partial z}+\overline{\tau}\frac{\partial}{\partial \overline{z}}.$$

Thus as differential operators $${\left(a\frac{\partial }{\partial x}+b\frac{\partial}{\partial y}\right)}^l={\left(\tau\frac{\partial }{\partial z}+\overline{\tau}\frac{\partial}{\partial \overline{z}}\right)}^l.$$

Now, we note that the operators
$a\frac{\partial }{\partial x}$ and $b\frac{\partial}{\partial y}$ commute with each other, as do $\tau\frac{\partial }{\partial z}$ and $\overline{\tau}\frac{\partial}{\partial \overline{z}}$, therefore, the claim follows after applying the previous lemma to both sides, dividing both sides by $k!$,  and applying the resulting operators to $f$.
\end{proof}

\begin{proof}[Proof of Poposition \ref{ztaylor}]
We begin the proof with the Taylor expansion of $f$ as a function of two variable $x$ and $y$ with the integral form of the remainder. This can be seen from the Taylor Expansion and Integral form of the remainder of the function $g(t)=f(tx,ty)$ (See, for example of \cite[Theorem 8.14]{fitz}) and the Chain Rule.
\begin{IEEEeqnarray}{rCl}f(z)&=&\sum_{l=0}^k \sum_{j=0}^l\frac{1}{j!(l-j)!}x^j y^{l-j}\frac{\partial^l f}{\partial x^j\partial y^{l-j}}(0)\nonumber \\
&&+\sum_{j=0}^{k+1} \frac {k+1}{j!(k+1-j)!}x^j y^{k+1-j}\int_0^1 (1-t)^k \frac {\partial^{k+1} f}{\partial x^j \partial y^{k+1-j}}(tz)dt\nonumber \end{IEEEeqnarray}

By Lemma \ref{taylorequation}, evaluating at $z=0$ and then taking $\tau=z=x+iy$, each term in the first sum over $k$ becomes
$$\sum_{j=0}^l \frac{1}{j!(l-j)!}z^j \overline{z}^{l-j}\frac{\partial^l f}{\partial z^j\partial \overline{z}^{l-j}}(0)$$

Likewise, after switching the order of summation and integration in the remaining sum, and applying Lemma \ref{taylorequation} (evaluating at $tz$), we get the remainder as
$$\sum_{j=0}^{k+1} \frac {k+1}{j!(k+1-j)!}z^j \overline{z}^{k+1-j}\int_0^1 (1-t)^k \frac {\partial^{k+1} f}{\partial z^j \partial \overline{z}^{k+1-j}}(tz)dt.$$

The desired result now follows.
\end{proof}

We end this Appendix with one final note. If $f$ is a function of several complex variables, then the Wirtinger derivatives with respect to the various variables are well-defined and commute with one another in the same way as the usual partial derivatives do \cite[p. 4]{hormander}.

\end{document}